\documentclass[reqno, a4paper]{amsart}

\usepackage[toc,page]{appendix}

\usepackage[utf8]{inputenc}
\usepackage{esint}
\usepackage{MnSymbol}
\usepackage{enumitem}
\usepackage[english]{babel}
\usepackage{amsmath}
\usepackage{thm-restate}
\usepackage{tikz-cd}
\usepackage[normalem]{ulem}
\usepackage{mathtools}
\usepackage{csquotes}
\usepackage{fontawesome}

\usepackage[
colorlinks,
pdfpagelabels,
pdfstartview = FitH,
bookmarksopen = true,
bookmarksnumbered = true,
linkcolor = blue,
plainpages = false,
hypertexnames = false,
citecolor = red] {hyperref}

\hypersetup{
linktoc=page
}
\hypersetup{breaklinks=true}

\usepackage[abbrev, alphabetic, msc-links]{amsrefs}
\usepackage{comment}

\usepackage[capitalize]{cleveref}
\crefname{enumi}{}{}
\crefname{equation}{}{}

\makeatletter
\def\@tocline#1#2#3#4#5#6#7{\relax
\ifnum #1>\c@tocdepth 
\else
\par \addpenalty\@secpenalty\addvspace{#2}%
\begingroup \hyphenpenalty\@M
\@ifempty{#4}{%
\@tempdima\csname r@tocindent\number#1\endcsname\relax
}{%
\@tempdima#4\relax
}%
\parindent\z@ \biglskip#3\relax \advance\biglskip\@tempdima\relax
\bigrskip\@pnumwidth plus4em \parfillskip-\@pnumwidth
#5\leavevmode\hskip-\@tempdima
\ifcase #1
\or\or \hskip 1em \or \hskip 2em \else \hskip 3em \fi%
#6\nobreak\relax
\dotfill\hbox to\@pnumwidth{\@tocpagenum{#7}}\par
\nobreak
\endgroup
\fi}
\makeatother

\newtheorem{theorem}{Theorem}
\newtheorem*{theorem*}{Theorem}
\newtheorem{proposition}{Proposition}[section]
\newtheorem{lemma}[proposition]{Lemma}
\newtheorem{corollary}[proposition]{Corollary}

\theoremstyle{definition}
\newtheorem{definition}[proposition]{Definition}
\newtheorem{remark}[proposition]{Remark}

\numberwithin{equation}{section}

\newcommand{\R}{\mathbb{R}}

\newcommand{\B}{\mathbf}
\renewcommand{\L}{\mathsf{L}}

\newcommand{\ud}{\, \mathrm{d}}
\newcommand{\diag}{\mathrm{diag}}

\def \R {\mathbb {R}}
\def \N {\mathbb {N}}
\def \Z {\mathbb {Z}}

\def \T {\mathbb{T}}

\def\supp{\operatorname{supp}}

\renewcommand{\tilde}{\widetilde}

\DeclarePairedDelimiter\abs{\lvert}{\rvert}
\DeclarePairedDelimiter\norm{\lVert}{\rVert}

\allowdisplaybreaks

\begin{document}
\title[Galerkin, Hermite and kinetic Fokker Planck equations]{
A Galerkin type method for kinetic
\\
Fokker Planck equations based
\\
on Hermite expansions}

\address{Benny Avelin \\Department of Mathematics, Uppsala University\\
S-751 06 Uppsala, Sweden}
\email{benny.avelin@math.uu.se}

\address{Mingyi Hou\\Department of Mathematics, Uppsala University\\
S-751 06 Uppsala, Sweden}
\email{mingyi.hou@math.uu.se}

\address{Kaj Nystr\"{o}m\\Department of Mathematics, Uppsala University\\
S-751 06 Uppsala, Sweden}
\email{kaj.nystrom@math.uu.se}

\author{Benny Avelin, Mingyi Hou and Kaj Nystr{\"o}m}

\begin{abstract}
  In this paper, we develop a Galerkin-type approximation, with quantitative error estimates, for weak solutions to the Cauchy problem for kinetic Fokker-Planck equations in the domain $(0, T) \times D \times \R^d$, where $D$ is either $\T^d$ or $\R^d$. Our approach is based on a Hermite expansion in the velocity variable only, with a hyperbolic system that appears as the truncation of the Brinkman hierarchy, as well as ideas from \cite{albrittonVariationalMethodsKinetic2021a} and additional energy-type estimates that we have developed. We also establish the regularity of the solution based on the regularity of the initial data and the source term.


  \medskip

\end{abstract}

\maketitle

\setcounter{equation}{0} \setcounter{theorem}{0}


\section{Introduction and statement of main results}
The kinetic Fokker-Planck equation is
\begin{align}\label{eq:kfp}
  \partial_t p = \nabla_v\cdot(\nabla_v p + vp) - v\cdot\nabla_x p + \nabla_x U(x)\cdot \nabla_v p,
\end{align}
where $x\in\R^d$ represents position, $v\in\R^d$ represents velocity, $t$ is time, and $U(x)$ is a potential that depends only on $x$. Subject to smoothness and growth conditions on $U(x)$, this equation describes the probability density $p(t,x,v)$ of the stochastic process $(x, v) = (x(t), v(t))$  where
\begin{align*}
  \begin{cases}
    \ud x = v \ud t, \\
    \ud v = - (v +\nabla_x U(x))\ud t + \sqrt{2}\ud B_t,
  \end{cases}
\end{align*}
and where  and $B_t$ is a standard Brownian motion in $\R^d$. It is well-known that
\begin{equation*}
  \rho(x, v) := e^{-\abs{v}^2/2-U(x)}
\end{equation*}
is the stationary solution of \cref{eq:kfp}. Defining
\begin{equation*}
  u(t,x,v):=p(t,x,v)/\rho(x, v),
\end{equation*}
it follows that $u$ satisfies the  conjugated form of \cref{eq:kfp}, i.e.
\begin{align}\label{eq:kfpconj}
  \partial_t u = \mathcal{L} u, \, \mathcal{L}:= (\Delta_v -v\cdot \nabla_v ) + (\nabla_x U\cdot\nabla_v - v\cdot \nabla_x ).
\end{align}
The operators in \cref{eq:kfp,eq:kfpconj} are sometimes referred to as Kramers-Smolu{-}chowski-Fokker-Planck operators.
The kinetic Fokker-Planck operator appears naturally in various fields of kinetic theory and statistical physics, including plasma physics, condensed matter physics, and more recently, in machine learning and in the optimization of deep neural networks using the method of stochastic gradient descent with momentum.

The purpose of this paper is to contribute to the study of Cauchy problems for the kinetic Fokker-Planck equation by constructing Galerkin-type approximations of weak solutions to the problem in \cref{ivp1} with quantitative error estimates. More precisely, we study the Cauchy problem
\begin{equation}\label{ivp1}
  \begin{cases}
    \partial_t u - \mathcal{L}u = f  &\mbox{in } (0,T)\times\Omega,\\
    \hfill u = g  &\mbox{on } \{t=0\}\times\Omega,
  \end{cases}
\end{equation}
where $u=u(t,x,v)$, $T\in (0,\infty)$, $\Omega = D\times\R^d$, and either $D=\T^d$, the $d$ dimensional torus, or $D=\R^d$.
In \cref{ivp1}, $g(x,v)$ is the initial data, and $f(t,x,v)$ is a given source term.
As we will explain, we restrict our attention to $D=\T^d$ and $D=\R^d$ for several reasons. One reason is the current lack of understanding regarding the existence and, in particular, uniqueness of weak solutions to the Cauchy-Dirichlet problem for the operator in \cref{eq:kfpconj}.

Various versions of the problem in \cref{ivp1} have been extensively studied from both theoretical and numerical perspectives over the past few decades. Notably, in  \cite{villaniHypocoercivity2009},  Villani established the existence and exponential decay to equilibrium, under assumptions on the potential $U=U(x)$. In  \cite{villaniHypocoercivity2009}, Villani also introduced the notion of {hypocoercivity} which essentially amounts to redeeming the initial lack of coercivity of the hypoelliptic operators at hand, by introducing auxiliary norms in which solutions are proved to dissipate and quantitative decay estimates can be deduced. 
An extension of the hypercoercivity framework for superlinear growth was studied by Dolbeault, Mouhot and Schmeiser in \cite{dolbeault2015hypocoercivity}.
In \cite{herauIsotropicHypoellipticityTrend2004a},  Herau and Nier gave explicit estimates of the rate of decay to the equilibrium using hypoelliptic techniques and by making the connection to the Witten Laplacian and its spectrum and spectral gaps.
Spectral gaps were further studied in a series of papers by Herau et al., including \cite{herauSemiclassicalAnalysisKramersFokkerPlanck2005,herauTunnelEffectKramersFokkerPlanck2008,herauTunnelEffectSymmetries2011a}.
In these papers the authors used semiclassical analysis to study the relationship between small eigenvalues and minimums of the potential $U$.
More recently, Bony et al.~\cite{bony2022eyringkramers} generalized the results from \cite{herauTunnelEffectSymmetries2011a} to general Fokker-Planck type operators admitting a Gibbs stationary distribution and established an Eyring-Kramers formula for the eigenvalues of the operator.
Different from the approaches mentioned above, Albritton et al.~\cite{albrittonVariationalMethodsKinetic2021a} proved the existence of weak solutions in $\T^d$ and exponential decay in an $\L^2$-setting using a variational approach based on energy methods.
Cao et al.~\cite{caoExplicitConvergenceRate2023} extended the results of \cite{albrittonVariationalMethodsKinetic2021a} to the whole space and developed a new method for determining the rate of decay to equilibrium in the $\L^2$ norm. Their approach involves expressing the estimate in terms of the Poincaré constant of the underlying space.

The work presented in this paper is inspired by the approach developed in \cite{albrittonVariationalMethodsKinetic2021a}, and our idea has been to develop Galerkin-type approximations, with quantitative error estimates, of weak solutions to the problem in \cref{ivp1} in the framework of \cite{albrittonVariationalMethodsKinetic2021a}.
In the case $D=\T^d$, our starting point is an ansatz for the solution to the problem in \cref{ivp1} in the form of a spectral expansion using Hermite polynomials in the velocity variable only, an approach which in the one-dimensional case was introduced by Risken in \cite{riskenFokkerplanckEquationMethods1996}. Through this ansatz we are led to an infinite system of equations, originally discovered in \cite{brinkmanBrownianMotionField1956}, in the  $(t,x)$ variables only.
By truncating this system or hierarchy of equations, a method that dates back to Grad \cite{gradKineticTheoryRarefied1949} and which is often referred to as the Brinkman hierarchy, we obtain a hyperbolic system  \cref{eq:hyperbolic_intro+} which plays a central role in our paper.
A byproduct of our approach and energy estimates, is that they yield regularity estimates for the solution that depend on the regularity of the initial data and the source term.
It is worth noting that a well-posedness result for the hyperbolic system was given in \cite{meyerCommentsGradProcedure1983b} and hypocoercivity estimates in \cite{blaustein2023discrete}. 
However, a novelty in our approach is that we use a vanishing viscosity method to prove the well-posedness of the hyperbolic system and to obtain the corresponding energy estimates. In addition, our approach gives an alternative proof of the existence and uniqueness of weak solutions to \cref{ivp1} established in \cite{albrittonVariationalMethodsKinetic2021a} in the case $D=\T^d$. Furthermore, compare to \cite{meyerCommentsGradProcedure1983b} and the numerical papers mentioned below,  we are able to quantify to what extent the (truncated) hyperbolic system approximates the Cauchy problem \cref{ivp1}  and we deduce error estimates depending essentially only on dimension $d$, the level of truncation, the potential $U$, the initial data and the source term, see \cref{Thm1}. In the case $D=\R^d$, we prove, under additional restrictions on the potential $U(x)$, that we can approximate the solution in  $\R^d\times\R^d$, by a series of bounded periodic solutions which each can be computed by the Galerkin method developed, see \cref{Thm2}. Furthermore,  also in this case we prove  the existence and uniqueness of weak solutions to \cref{ivp1}.

Concerning the literature devoted to numerical schemes to the problem in \cref{ivp1}, 
Risken developed a matrix continued-fraction method in \cite{riskenFokkerplanckEquationMethods1996} based on the spectral expansion.
In \cite{fokCombinedHermiteSpectralfinite2002a}, Fok et al.~gave detailed explicit and implicit discretization schemes for the hyperbolic system \cref{eq:hyperbolic_intro+} in the case $d=1$.
A similar numerical scheme for the periodic case can be found in \cite{chaiMixedGeneralizedHermiteFourier2018}.
In addition, for the Fokker-Planck-Landau equation, the nonlinear version of equation \cref{ivp1}, a numerical scheme based on Hermite expansion can be found in \cite{liHermiteSpectralMethod2021a}.
Numerical analysis for various other types of kinetic equations, e.g.~Boltzmann equation, using the Galerkin method can be found in e.g.~\cite{sarnaConvergenceAnalysisGrad2020,schmeiserConvergenceMomentMethods1998,gambaGalerkinPetrovApproach2018} and references therein.
In particular, in \cite{sarnaConvergenceAnalysisGrad2020}, the authors analyze strong solutions to linear kinetic equations in spatially bounded domains using Hermite polynomials and obtain $\L^2$ error estimates for the approximation which depends on the regularity of the solution. 
However, in the current paper, we obtain explicit $\mathrm{H}^k$ error estimates for Galerkin approximations of weak solutions to the kinetic Fokker-Planck equation that only depend on the regularity of the initial data and source term, see \cref{thm1:error}. 

Finally, we also note that different spectral methods may be used to solve the kinetic Fokker-Planck equations, for example, a generalized Laguerre-Legendre spectral method can be found in \cite{guoCompositeGeneralizedLaguerreLegendre2009a}. In addition, we refer to \cite{pavliotisStochasticProcessesApplications2014a} for a good introduction to Monte Carlo methods and the simulation of probability densities. 

\subsection{Function spaces and weak solutions} 
Points in $\R\times\R^d\times\R^d$ will be denoted by $(t,x,v)$ where  $x\in \R^d$ is position, $v\in\R^d$ is velocity, and $t$ is time. We introduce  Gibbs measures $\mu$ and $\eta$ on $\R^d$,
\begin{align}\label{measures}
  \ud \mu = \ud\mu(v) = e^{-\abs{v}^2/2}\ud v, \quad \ud \eta=\ud\eta(x) = e^{-U(x)}\ud x,
\end{align}
and we will consistently assume that
$e^{-{U(x)}}\in \mathrm{L}^1(\R^d,\ud x)$.

Our function spaces are defined with respect to the measures $\mu$ and $\eta$ introduced in \cref{measures}. Given $D \subset \R^d$, and $\zeta\in\{\mu,\eta\}$,
we let $\L^2_\zeta(D)$ be the Hilbert space with inner product
\begin{equation*}
  (u,w)_{\L^2_\zeta}:=\int_D u w \, \ud \zeta,
\end{equation*}
and norm
\begin{equation*}
  \|u\|_{\L^2_\zeta} := (u,u)_{\L^2_\zeta}^{1/2}=\bigl ( \int_D u^2 \, \ud \zeta \bigr )^{1/2}.
\end{equation*}
We let $\mathrm{H}^1_\zeta(D)$ be the first order Sobolev space with inner product
\begin{equation*}
  ( u,w )_{\mathrm{H}^1_\zeta}:=(u,w)_{\L^2_\zeta}+(\nabla_v u ,\nabla_v w )_{\L^2_\zeta},
\end{equation*}
and with norm
\begin{equation*}
  \|u\|_{\mathrm{H}^1_\zeta} := \bigl(\|u\|^2_{\L^2_\zeta} + \|\nabla u\|^2_{\L^2_\zeta}\bigr)^{1/2}.
\end{equation*}
We also let $(\cdot,\cdot)_{\mathrm{H}^{-1}_\zeta, \mathrm{H}^1_\zeta}$ denote the pairing between $\mathrm{H}^{-1}_\zeta $ and $\mathrm{H}^1_\zeta$ where $\mathrm{H}^{-1}_\zeta$ is the dual space of $\mathrm{H}^1_\zeta$.
It is well known that under rather weak conditions on the tail of the measure $\eta$, the spaces $\L^2_\zeta(\R^d)$, $\mathrm{H}^1_\zeta(\R^d)$, $\L^2_\zeta(\T^d)$, $\mathrm{H}^1_\zeta(\T^d)$ are the closures of the space of compactly supported smooth functions, $C_0^\infty(\R^d)$ and $C_0^\infty(\T^d)$ in the corresponding norms, see
e.g.~\cite{adamsSobolevSpaces1975a,tolleUniquenessWeightedSobolev2012b}.
The spaces $\mathrm{H}^k_\mu(\R^d)$ and $\mathrm{H}^k_\eta(\T^d)$, for $k\geq 2$, are defined analogously. 

Given an arbitrary Banach space $X$ we define the Bochner space $\L^2_\zeta(D;X)$ to be the set of all measurable functions $f:D \to X$ such that
\begin{equation*}
  \|f\|_{\L^2_\zeta(D;X)} := \bigl ( \int_D \|f\|_X^2 \, \ud\zeta \bigr )^{1/2} < \infty.
\end{equation*}
The space, $\L^2((0,T);X)$ is defined analogously, and hence, for instance, the space $\L^2((0,T);\L^2_\eta(\T^d;\L^2_\mu(\R^d)))$ has a well-defined meaning.
Given $D\subset\R^d$, and $T>0$, we will, if unambigous, drop either some or all of the domains of definition in the Bochner spaces $\L^2(\L^2_\eta(\mathrm{H}^{1}_\mu))=\L^2((0,T);\L^2_\eta (D;\mathrm{H}^{1}_\mu(\R^d)))$ and
$\L^2(\L^2_\eta(\mathrm{H}^{-1}_\mu))=\L^2((0,T);\L^2_\eta (D;\mathrm{H}^{-1}_\mu(\R^d)))$, here the subscripts $\eta$ and $\mu$ indicated the measure with respect to which integration is performed. 

Furthermore, as in \cite{albrittonVariationalMethodsKinetic2021a}, given $D \subset \R^d$ and $T>0$, we define $\mathrm{H}^1_{\text{kin}}:=\mathrm{H}^1_{\text{kin}}((0,T)\times D\times \R^d)$ as the set
\begin{align*}
  \left \{u:u \in \L^2(\L^2_\eta(\mathrm{H}^{1}_\mu)), \partial_t u + v\cdot\nabla_x u -\nabla_x U\cdot \nabla_v u \in \L^2(\L^2_\eta(\mathrm{H}^{-1}_\mu)) \right \}.
\end{align*}
We remark that, equipped with the norm
\begin{align*}
  \|u\|_{\mathrm{H}^1_{\text{kin}}}:= \|u\|_{\L^2(\L^2_\eta(\mathrm{H}^1_\mu))}
  +
  \|\partial_t u + v\cdot\nabla_x u -\nabla_x U\cdot \nabla_v u\|_{\L^2(\L^2_\eta (\mathrm{H}^{-1}_\mu))},
\end{align*}
$\mathrm{H}^1_{\text{kin}}$ is a Banach space.

\begin{definition}\label{def:weak}
  Given a domain $D \subset \R^d$ and $T > 0$, if $g\in \L^2_\eta(D;\allowbreak \L^2_\mu)$ and $f\in \L^2((0,T);\L^2_\eta(D;\mathrm H^{-1}_\mu))$, we say that a function $u\in \mathrm{H}^1_{\mathrm{kin}}$ is a weak solution to \cref{ivp1} if
  \begin{equation*}
    \underset{(0,T) \times D}{\iint} (\nabla_v u, \nabla_v \varphi)_{\L^2_\mu}\ud\eta\ud t
    =
    \underset{(0,T) \times D}{\iint} (f-\partial_t u- v\cdot\nabla_x u +\nabla_x U\cdot\nabla_v u, \varphi )_{\mathrm{H}^{-1}_\mu,\mathrm{H}^1_\mu}\ud\eta\ud t,
  \end{equation*}
  holds for all $\varphi\in \L^2((0,T);\L^2_\eta(D;\mathrm{H}^1_\mu))$ and if
  \begin{equation*}
    u(0,x,v) = g(x,v)\mbox{ for }(x,v)\in D\times\R^d.
  \end{equation*}
\end{definition}

\begin{remark}
  The initial condition in \cref{def:weak} is well-defined as any function $u\in  \mathrm{H}^1_{\mathrm{kin}}((0,T)\times D\times \R^d)$ can be identified with an element of
  $C([0,T];\allowbreak \L^2_\eta(D;\allowbreak \L^2_\mu(\R^d)))$, see \cite[Lemma 6.12]{albrittonVariationalMethodsKinetic2021a}.
\end{remark}
To formulate our main results we also introduce, for $k\geq 1$, $D \subset \R^d$, and $T > 0$, the function space $\mathrm{H}^k_{t,x,v}:=\mathrm{H}^k_{t,x,v}((0,T)\times D\times\R^d) $
which
is defined as
\begin{multline*}
  \mathrm{H}^k_{t,x,v}((0,T)\times D\times\R^d):=\big\{u: \partial_t^i\nabla_x^j\nabla_v^{2l} u \in \L^2(L^2_\eta(\L^2_\mu)), \\
  \forall i,j,k \in \N \text{ s.t. } i+j+l \leq k \big\}.
\end{multline*}
We equip this space with the norm
\begin{equation*}
  \norm{ u }^2_{\mathrm{H}^k_{t,x,v}}:= \sum_{i+j+l = 0}^k \norm{\partial_t^i \nabla_x^j \nabla_v^{2l} u}^2_{\L^2(\L^2_\eta(\L^2_\mu))}.
\end{equation*}
The spaces $\mathrm{H}^k_{t,x}:=\mathrm{H}^k_{t,x}((0,T)\times D)$ and $\mathrm{H}^k_{x,v}:=\mathrm{H}^k_{x,v}(D\times\R^d)$ are defined analogously.

We finally remark that, if $U(x) \in C^\infty(\T^d)$, then
\begin{equation*}
  \mathrm{H}^k_{t,x,v}((0,T)\times\T^d\times\R^d)\subset \mathrm{H}^1_{\text{kin}}((0,T)\times \T^d\times \R^d),
\end{equation*}
for all $k\geq 1$, but this is not the case if $\T^d$ is replaced by $\R^d$.

\subsection{Statement of main results} 
Our main results concern the geometric situations $D=\T^d$ and $D=\R^d$.

In the case $D=\T^d$, we consider the problem in \cref{ivp1} and the ansatz
\begin{align} \label{eq:ansatz}
  u_m(t,x,v) = \sum_{\abs{\alpha}=0}^m c_m^\alpha(t,x)\Psi_\alpha (v),
\end{align}
where $\{\Psi_\alpha(v)\}$ are normalized Hermite polynomials in $\R^d$ introduced in \cref{eq:hermite}, with the coefficients $\{c_m^\alpha\}$ satisfying the hyperbolic system stated in \cref{eq:hyperbolic_intro+}. In this case our main result can be stated as follows.
\begin{theorem}\label{Thm1}
  Let $D=\T^d$, $T > 0$, and let $k \in \Z$ be such that $k\geq2$. If we assume that
  \begin{equation}\label{eq:assump:um1}
    \begin{cases}
      U(x) \in C^\infty(\T^d),\ g\in \mathrm{H}^{k}_{x,v}(\T^d\times\R^d),\ f\in \mathrm{H}^{k}_{t,x,v}((0,T)\times\T^d\times\R^d). \\
      \nabla_x^{k+1} g\in \L^2_\eta(\T^d;\L^2_\mu(\R^d)), \\
      \partial_t^i \nabla_x^{j} f \in \L^2((0,T); \L^2_\eta(\T^d;\L^2_\mu(\R^d))), \text{ whenever } i+j =k+1,
    \end{cases}
  \end{equation}
  then there exists, for each $m\in\Z_+$, a unique function
  \begin{equation*}
    u_m\in\mathrm{H}^k_{t,x,v}((0,T)\times\T^d\times\R^d)
  \end{equation*}
  of the form \cref{eq:ansatz} with $\{c_m^\alpha\}$ satisfying the hyperbolic system \cref{eq:hyperbolic_intro+}. Also, there exists a constant
  $C=C(d,T,\norm{\nabla_x^2 U}_\infty,\dots, \norm{\nabla_x^{k+1}U}_\infty,k) \geq 1$, such that
  \begin{equation*}
    \norm{u_m}^2_{\mathrm{H}^k_{t,x,v}} \leq C (\norm{g}^2_{\mathrm{H}^k_{x,v}} + \norm{f}^2_{\mathrm{H}^k_{t,x,v}}).
  \end{equation*}
  Furthermore, after passing to a subsequence,
  \begin{equation*}
    u_m\rightharpoonup u\in \mathrm{H}^{k}_{t,x,v}((0,T)\times\T^d\times\R^d),
  \end{equation*}
  where $u$ is the unique weak solution to \cref{ivp1} in the sense of \cref{def:weak}, and 
  {
    for all $0\leq l\leq k-2$,
    \begin{equation}\label{thm1:error}
      \norm{u-u_m}^2_{\mathrm{H}^l_{t,x,v}} \leq C \left(\frac{1}{(1+m)^{(k-l-2)+\frac{1}{2}}} \right)^2 \bigl( \norm{g}^2_{\mathrm{H}^k_{x,v}} + \norm{f}^2_{\mathrm{H}^k_{t,x,v}}\bigr),
    \end{equation}
    where $C = C(d,T,\norm{\nabla_x^2 U}_\infty,\dots, \norm{\nabla_x^{k+1}U}_\infty,k) \geq 1$.
    }
\end{theorem}
\begin{remark}
In the case where $l=0$ and $k=2$, we observe that the decay of the error $\|u-u_m\|_{\L^2}$ is $1/\sqrt{m}$. To understand why this is the case, note first that the natural regularity required to achieve boundedness of the difference is to have control over $\|\nabla_v \nabla_x u\|_{\L^2(\L_\eta^2(\L^2_\mu))}$, as $(\partial_t - \mathcal{L}) (u-u_m)$ contains only the $m+1$-th term in the Hermite expansion of the mixed derivative terms $\nabla_v \nabla_x u$. From this, we can see that estimating the error is equivalent to estimating the decay of the expansion coefficients. We can bound the decay of the tail by the $\mathrm{H}^2_{t,x,v}$ norm of $u$ with a rate of $1/\sqrt{m}$, which is, in this sense, sharp. See \cref{rmk:sharp} for a more detailed discussion.

We also remark that the proof can be modified to obtain fractional regularity, i.e. $k$ can be fractions s.t. $k-l > 3/2$, and the definition of $\mathrm{H}^k_{t,x,v}$ should be changed correspondingly.
\end{remark}

Using \cref{Thm1} and an approximation argument, we can also deduce the following well-posedness result.
\begin{corollary}\label{cor:l2initial}
  Let $D=\T^d$, $T > 0$ and assume that
  \begin{equation*}
    U(x)\in C^\infty(\T^d),\ g\in \L^2_\eta(\T^d; \L^2_\mu(\R^d)),\ f \in \L^2((0,T); \L^2_\eta(\T^d;\L^2_\mu(\R^d))).
  \end{equation*}
  Then there exists a unique weak solution $u\in \mathrm{H}^1_{\mathrm{kin}}((0,T)\times\T^d\times\R^d)$ to the initial value problem in \cref{ivp1}, and there exists a constant $C = C(d,\allowbreak T,\allowbreak \norm{\nabla_x^2 U}_\infty,\allowbreak \norm{\nabla_x^{3}U}_\infty) \geq 1$, such that
  \begin{align*}
    \norm{u}_{\mathrm{H}^1_{\mathrm{kin}}} &\leq C\bigl( \norm{g}_{\L^2_\eta(\L^2_\mu)} +
    \norm{f}_{\L^2(\L^2_\eta(\L^2_\mu))}\bigr).
  \end{align*}
\end{corollary}

{Let now $Q_R$ denote the cube with sidelength $2R$ centered at the origin. We consider the following family of truncated problems with periodic boundary conditions
\begin{align}\label{eq:ivp4intro}
  \begin{cases}
    (\partial_t - \mathcal{L}_R)u_R = f & \text{ in } (0,T)\times Q_R\times\R^d,\\
    \hfill u_R = g & \text{ on } \{t=0\}\times Q_R\times \R^d,
  \end{cases}
\end{align}
where $\mathcal{L}_R$ is the operator $\mathcal{L}$ from \cref{eq:kfpconj} with the potential $U$ replaced by $U_R (x) := U(x)\varphi_R(x)$, where $\chi_{Q_{R/2}} \leq \varphi_R \leq \chi_{Q_R}$ is a smooth cutoff function. Further details can be found in \cref{sec:global}.
We now state our next main result valid for $D = \R^d$.
}
\begin{theorem}\label{Thm2}
  Let $D=\R^d$ and $T > 0$, assume that $U(x)$ has the form
  \begin{equation}\label{assump:global:potential}
    U(x) = a |x|^2 + P(x),\ a\in \R_+,\ P(x)\in C_0^\infty(\R^d),
  \end{equation}
  and  that
  \begin{equation}\label{assump:global:potential+}
    g\in C_0^\infty(\R^{2d}),\ f\in C^\infty_0([0,T]\times\R^{2d}),
  \end{equation}
  where $C_0^\infty(X)$ denotes the space of all compactly supported smooth functions on $X$.
  Then there exists an $R_0 = R_0(U, g, f)$ such that for $R>R_0$,
  and after passing to a subsequence,
  \begin{equation*}
    u_R\rightharpoonup u\in \mathrm{H}^1_{\mathrm{kin}}((0,T)\times\R^d\times\R^d) \textrm{ as } R\to \infty,
  \end{equation*}
  where 
  $u_R$ solves \cref{eq:ivp4intro} and
  $u$ is the unique weak solution to \cref{ivp1}.
\end{theorem}

Using \cref{Thm2} and an approximation argument, we can also deduce the following  well-posedness result.
\begin{corollary}\label{cor:global:l2initial}
  Let $D$, $U$, and $T$ be as in \cref{Thm2}. If $g\in \L^2_\eta(\R^d;\L^2_\mu(\R^d))$, and $f \in \L^2((0,T); \L^2_\eta(\R^d;\L^2_\mu(\R^d)))$,
  then there exists a unique weak solution $u\in \mathrm{H}^1_{\mathrm{kin}}((0,T)\times\R^d\times\R^d)$ to the initial value problem \cref{ivp1}.
\end{corollary}

\subsection{Outline of the proofs} \label{sec:outline}
To prove \cref{Thm1} we deduce that
\begin{equation*}
  (\partial_t - \mathcal{L}) u_m = f_m + {E}_m
\end{equation*}
where $f_m$ is a finite dimensional projection of $f$ and ${E}_m$ is considered as an error term, see \cref{eq:Lum}.
We then prove that $\{u_m\}$ and $\{\partial_t u_m + v\cdot\nabla_x u_m -\nabla_x U\cdot \nabla_v u_m\}$ are uniformly bounded sequences in $\L^2((0,T);\L^2_\eta(\T^d;\mathrm{H}^1_\mu))$ and $\L^2((0,T);\allowbreak \L^2_\eta(\T^d;\allowbreak  \L^2_\mu))$ respectively, see \cref{thm:regularity}.
Furthermore, we prove that ${E}_m \to 0$ strongly in $\L^2((0,T);\L^2_\eta(\T^d; \L^2_\mu))$, see \cref{lem:Em}. Together, this allows us to conclude that
\begin{equation*}
  u_m \rightharpoonup u
\end{equation*}
in $\L^2((0,T);\L^2_\eta(\T^d; \mathrm{H}^1_\mu))$,  that
\begin{equation*}
  \partial_t u_m + v\cdot\nabla_x u_m -\nabla_x U\cdot \nabla_v u_m \rightharpoonup \partial_t u + v\cdot\nabla_x u -\nabla_x U\cdot \nabla_v u
\end{equation*}
in $\L^2((0,T);\allowbreak \L^2_\eta(\T^d;\allowbreak  \mathrm{H}^{-1}_\mu))$, and that
\begin{equation*}
  {E}_m \rightharpoonup 0
\end{equation*}
in $\L^2((0,T);\L^2_\eta(\T^d; \mathrm{H}^{-1}_\mu))$.
The error estimate in the theorem follows from the energy estimates for $u_m$, which are equivalent to the energy estimates for the coefficients $\{c_m^\alpha\}_\alpha$, and a delicate bootstrapping argument (see \cref{thm:regularity}).

To prove \cref{Thm2}, we construct a sequence of bounded periodic problems \cref{eq:ivp4intro} indexed by the truncation scale $R$, using certain cutoff functions. We denote the solutions to these problems by ${u_R}$, and after passing to a subsequence, we show that these solutions converge to a solution $u\in\mathrm{H}^{1}_\mathrm{kin}((0,T)\times\R^d\times\R^d)$ for the problem \cref{ivp1}. The proof involves using the logarithmic Sobolev inequality for the measure $\ud\eta\ud\mu$ and an entropy-type inequality from \cite{ledouxConcentrationMeasureLogarithmic1999a}.

\subsection{Organization of the paper.}
In \cref{sec:pre}, we introduce essential notations and cover preliminaries.

In \cref{sec:hsys}, we establish well-posedness and regularity estimates of the hyperbolic system that appears as the truncation of the Brinkman hierarchy.

In \cref{sec:uniformestimate}, we establish the essential energy estimates for the Galerkin approximations ${u_m}$.

In \cref{sec:bddresult}, we prove both \cref{Thm1} and \cref{cor:l2initial}.

Finally, in \cref{sec:global}, we prove \cref{Thm2} and \cref{cor:global:l2initial}.

\section*{Acknowledgments}
B.~Avelin and M.~Hou were supported by [Swedish Research Council dnr: 2019-04098]. K.~Nystr\"{o}m was supported by [Swedish Research Council dnr: 2022-03106]. The authors wish to thank the referees for their careful reading of the original manuscript which helped us to improve the presentation.

\section{Preliminaries}\label{sec:pre}

\subsection{Operators and their adjoints}  
Consider, for $i=1,...,d$, the operators $\partial_{v_i},\partial_{x_i}$. The formal adjoints of these operators in $\L^2_\mu$, $\L^2_\eta$ respectively, are
\begin{align}\label{eq:formaldual}
  \partial_{v_i}^* = -\partial_{v_i} + {v_i}, \quad \partial_{x_i}^* = -\partial_{x_i} + {\partial_{x_i}U}.
\end{align}
The corresponding nabla operators $\nabla_{v},\nabla_{v}^\ast,\nabla_{x},\nabla_{x}^\ast$ are defined in the canonical way.
Using this notation can rewrite the operator $\mathcal{L}$ in \cref{eq:kfpconj} as
\begin{equation*}
  \mathcal{L} = - \nabla_v^*\cdot\nabla_v -   \bigl( \nabla_v^*\cdot\nabla_x -\nabla_v\cdot\nabla_x^*\bigr).
\end{equation*}

\subsection{Hermite basis} 
In the case $d=1$, then the so-called Hermite polynomials form a basis for $\L^2_\mu(\R)$ as well as for $\mathrm{H}^1_\mu(\R)$. We will deal extensively with this basis in this paper and we here record some of its useful properties. The $k$-th Hermite function or polynomial is defined, see \cite{wienerFourierIntegralCertain1988},
\begin{equation*}
  h_k(v) := (-1)^k e^{v^2/2} \frac{\ud^k}{\ud v^k} e^{-v^2/2}.
\end{equation*}
We introduce the normalized Hermite functions as
\begin{equation*}
  \psi_k(v) := \frac{1}{Z_k} h_k(v),\quad Z_k = \bigl(\int_\R h_k^2(v) e^{-\frac{v^2}{2}}\ud v \bigr)^{1/2}.
\end{equation*}
We also introduce the Ornstein-Uhlenbeck operator as
\begin{align}\label{eq:ornuhl}
  K = \partial_{v}^* \partial_v = -\partial_v^2 + v\partial_v.
\end{align}
Then
\begin{equation*}
  K\psi_k = k\psi_k,
\end{equation*}
i.e.~$\{\psi_k\}$ are eigenfunctions to the  one-dimension Ornstein-Uhlenbeck operator $K$. The methodology outlined in this paper rests on the following elementary and well-known result.
We list the proof here for completeness.
\begin{lemma} \label{lem:simultaneous}
  The normalized Hermite functions $\{\psi_k\}_{k=0}^\infty$ is an \textit{orthonormal basis} in $\L^2_\mu(\R)$. Furthermore, $\{\psi_k\}_{k=0}^\infty$ is an \textit{orthogonal basis}  in $\mathrm{H}^1_\mu(\R)$.
\end{lemma}
\begin{proof}
For the first conclusion, it is sufficient to note that it is well known that $\{\psi_k\}_{k=0}^\infty$ is dense in $\L^2_\mu(\R)$, see \cite[Chapter 8, Theorem 1]{wienerFourierIntegralCertain1988}, and that the functions in $\{\psi_k\}_{k=0}^\infty$ are normalized and pairwise orthogonal in $\L^2_\mu(\R)$. For the second conclusion, we note that 
\begin{align}\label{calco}
  ( \psi_k, \psi_l )_{\mathrm{H}^1_\mu} &= (\psi_k,\psi_l)_{\L^2_\mu} + (\psi_k',  \psi_l')_{\L^2_\mu}\notag \\
  &= (\psi_k,\psi_l)_{\L^2_\mu} + (\psi_k, K \psi_l)_{\L^2_\mu} = (\psi_k, (1+l)\psi_l)_{\L^2_\mu} = (1+l)\delta_{kl},
\end{align}
i.e.~$\{\psi_k\}_{k=0}^\infty$ is orthogonal in $\mathrm{H}^1_\mu(\R)$. To prove that $\{\psi_k\}_{k=0}^\infty$ is dense in $\mathrm{H}^1_\mu(\R)$, suppose that there exists $u\in \mathrm{H}^1_\mu(\R)$, not identically zero, such that  $(u,\psi_k )_{\mathrm{H}^1_\mu} = 0$ for all $k=0,1,2,\dots$.
As in \cref{calco} this implies that
\begin{equation*}
  (u, (1+k)\psi_k)_{\L^2_\mu}=0,
\end{equation*}
for all $k=0,1,2,\dots$. Therefore $u$ must be identically zero in $\L^2_\mu(\R)$ since $\{\psi_k\}_{k=0}^\infty$ is dense in $\L^2_\mu(\R)$. Hence $\{\psi_k\}_{k=0}^\infty$ is simultaneous basis for $\L^2_\mu(\R)$ and $\mathrm{H}^1_\mu(\R)$.
\end{proof}

Given $d\in\Z$, $d\geq 1$, and a multi-index $\alpha=(\alpha_1, \dots, \alpha_d)$, we introduce the Hermite basis for $\R^d$,
\begin{align}\label{eq:hermite}
  \Psi_\alpha(v) = \prod_{i=1}^d \psi_{\alpha_i}(v_i).
\end{align}
Then by construction $\norm{\Psi_\alpha}_{\L^2_\mu(\R^d)} = 1$. Furthermore, for any multi-indices $\alpha,\alpha'$, we have
\begin{equation*}
  (\Psi_\alpha,\Psi_{\alpha'})_{\L^2_\mu} = \delta_{\alpha\alpha'},
\end{equation*}
and, for each $i=1,\dots,d$,
\begin{equation*}
  \partial_{v_i}^\ast\partial_{v_i} \Psi_\alpha = \alpha_i \Psi_\alpha.
\end{equation*}
It is immediate by \cref{lem:simultaneous} that $\{\Psi_\alpha\}_{|\alpha|=0}^\infty$, where $\abs{\alpha} = \sum_{i=1}^d \alpha_i$, forms an orthonormal basis for $\L^2_\mu(\R^d)$, and an orthogonal basis for $\mathrm{H}^1_\mu(\R^d)$.

Let, for each $i=1,\dots,d$,  $\mathbf{e}_i$ be the multi-index with a $1$ in position $i$ and zero elsewhere. Then the following useful and well-known recurrence relations can easily be verified
\begin{align}\label{eq:Hermiterelations}
  \begin{split}
    & \partial_{v_i} \Psi_\alpha (v)= \sqrt{\alpha_i}\Psi_{\alpha-\mathbf{e}_i}(v), \quad \alpha_i > 0,\\
    & \partial_{v_i}^\ast \Psi_\alpha (v) = \sqrt{\alpha_i + 1}\Psi_{\alpha+\mathbf{e}_i}(v),\\
    & v_i \Psi_\alpha = \sqrt{\alpha_i + 1}\Psi_{\alpha+\B{e}_i} + \sqrt{\alpha_i}\Psi_{\alpha-\B{e}_i}.
  \end{split}
\end{align}
Furthermore, using the above relations and orthogonality, we can write
\begin{equation*}
  \|\Psi_\alpha\|^2_{\mathrm{H}^1_\mu(\R^d)} = (\Psi_\alpha,\Psi_\alpha)_{\L^2_\mu}+ (\nabla_v \Psi_\alpha,\nabla_v \Psi_\alpha)_{\L^2_\mu} = 1+\abs{\alpha}.
\end{equation*}

\subsection{Sobolev spaces with respect to Gaussian measure}
For any function $u\in \mathrm{H}^1_\mu(\R^d)$ we can use \cref{lem:simultaneous} to expand it in terms of the Hermite basis \cref{eq:hermite},
\begin{equation*}
  u = \sum_{\abs{\alpha}\geq 0} c^\alpha(u) \Psi_\alpha, \quad c^\alpha(u) = (u, \Psi_\alpha)_{\L^2_\mu}.
\end{equation*}
The expansion converges in $\mathrm{H}^1_\mu(\R^d)$ with
\begin{equation*}
  \|u \|^2_{\mathrm{H}^1_\mu(\R^d)} = \norm{u}^2_{\L^2_\mu(\R^d)} + \norm{\nabla_v u}^2_{\L^2_\mu(\R^d)} = \sum_{\abs{\alpha}\geq 0} (1+\abs{\alpha})|c^\alpha(u)|^2.
\end{equation*}
For $u\in \mathrm{H}^k_\mu(\R^d)$ we also have
\begin{align} \label{eq:fractional_norm}
  \|u\|^2_{\mathrm{H}^{k}_\mu(\R^d)}  = \sum_{\alpha} (1+\abs{\alpha})^{k} |(u,\Psi_\alpha)_{\L^2_\mu}|^2 .
\end{align}

\section{The associated hyperbolic system}\label{sec:hsys}
In this section we consider the ansatz \cref{eq:ansatz} and derive a hyperbolic system for the coefficients $c^\alpha$. Specifically, for $m\geq 1$ fixed, we let
\begin{equation*}
  c_m^{\alpha-\mathbf{e}_i} := 0\mbox{ if }\alpha_i = 0,\ c_m^{\alpha+\mathbf{e}_i} := 0\mbox{ if }|\alpha+\mathbf{e}_i| > m,
\end{equation*}
and applying $\partial_t-\mathcal{L}$ to $u_m$ from \cref{eq:ansatz} we get
\begin{align}\label{eq:um:expansion}
  (\partial_t - \mathcal{L} )u_m(t,x,v) =& \sum_{\abs{\alpha}=0}^m \bigl( \partial_t c_m^\alpha(t,x) \Psi_\alpha(v) +  \abs{\alpha} c_m^\alpha(t,x) \Psi_\alpha(v)\bigr)\notag\\
  &- \sum_{\abs{\alpha}=0}^m \bigl(\sum_{i=1}^d \sqrt{\alpha_i}\partial^*_{x_i} c_m^\alpha(t,x) \Psi_{\alpha-\mathbf{e}_i}(v)\bigr)\notag\\
  &+\sum_{\abs{\alpha}=0}^m \bigl(\sum_{i=1}^d\sqrt{\alpha_i+1}\partial_{x_i} c_m^\alpha(t,x) \Psi_{\alpha+\mathbf{e}_i}(v)\bigr).
\end{align}
Expand the right hand side $f$ of \cref{ivp1} in the Hermite basis
\begin{equation*}
  f(x,v) = \sum_{|\alpha|=0}^\infty f^\alpha(x) \Psi_\alpha(v),
\end{equation*}
where $f^\alpha(x) := (f(x,\cdot),\Psi_\alpha)_{\L^2_\mu}$ for all $\alpha$ such that $\abs{\alpha}\geq 0$.
Then, setting the right hand side of \cref{eq:um:expansion} to $\sum_{|\alpha|=0}^m f^\alpha \Psi_\alpha$ in the space spanned by $\{\Psi_{\alpha}\}_{|\alpha|= 0}^m$ we get the hyperbolic system
\begin{multline}\label{eq:hyperbolic_intro+}
  f^{\alpha}(t,x)= \ \partial_t c_m^\alpha(t,x)  +  \abs{\alpha} c_m^{\alpha}(t,x) \\
  - \sum_{i=1}^d \sqrt{\alpha_i+1} \partial^*_{x_i} c_m^{\alpha+\mathbf{e}_i}(t,x)+ \sum_{i=1}^d\sqrt{\alpha_i}\partial_{x_i} c_m^{\alpha-\mathbf{e}_i}(t,x),
\end{multline}
where $0\leq \abs{\alpha}\leq m$,   and  $\partial_{x_i}^*$ is defined in \cref{eq:formaldual}.
This is the truncated hyperbolic system that appeared in \cite{brinkmanBrownianMotionField1956,riskenFokkerplanckEquationMethods1996}.
Assuming that we have a solution to \cref{eq:hyperbolic_intro+}, and inserting it into \cref{eq:um:expansion} we get
\begin{align}\label{eq:um:expansion_err}
  (\partial_t - \mathcal{L} )u_m(t,x,v) = \sum_{|\alpha| = 0}^m f^\alpha \Psi_\alpha + \sum_{\abs{\alpha} = m} \sum_{i=1}^d \sqrt{\alpha_i+1} \partial_{x_i} c_m^\alpha \Psi_{\alpha+\mathbf{e}_i} .
\end{align}
Thus, in order to show that \cref{eq:ansatz} is a reasonable ansatz, we need to establish existence and regularity for solutions to \cref{eq:hyperbolic_intro+}.

To write \cref{eq:hyperbolic_intro+} in a more compact form, we first order all multi-indices $0 \leq |\alpha| \leq m$ into an array $\{\alpha^i\}$ such that  $|\alpha^i| \leq |\alpha^j|$ for all $i \leq j$. Let $\mathbf{c}_m$ denote the column vector $\B{c}_m :=(c_m^{\alpha^1},\ldots,c_m^{\alpha^n})$ where
\begin{equation*}
  n = \#\lbrace \alpha \textrm{ multi-index s.t. } 0\leq \abs{\alpha}\leq m \rbrace.
\end{equation*}
Using this notation we prove the following auxiliary lemma.
\begin{lemma} \label{lem:skew}
  The system in \cref{eq:hyperbolic_intro+} can be written as
  \begin{align*}
    \partial_t \mathbf{c}_m +  A \mathbf{c}_m  - \sum_{i=1}^d B_i \partial_{x_i}^\ast \mathbf{c}_m + \sum_{i=1}^d B_i^T \partial_{x_i} \mathbf{c}_m = \B{f}_m,
  \end{align*}
  where $\B{f}_m = (f^{\alpha^1},\dots,f^{\alpha^n})$,  $f^{\alpha^j} := (f,\Psi_{\alpha^j})_{\L^2_\mu}$,  $A = \diag(\{|\alpha^j|\}_{j=0}^1)$,  and each $B_i$ is constant $n \times n$ upper triangular matrices with zeros on the diagonal.
\end{lemma}
\begin{proof}
  Fix $i \in \{1,\ldots,d\}$. By the definition of $B_i$ we have, for $1 \leq j \leq n$,
  \begin{equation*}
    \{B_i \partial_{x_i}^\ast \mathbf{c}_m\}_{j} = \sqrt{\alpha_i^j+1}\  \partial^*_{x_i} c_m^{\alpha^j+\mathbf{e}_i},
  \end{equation*}
  if $\alpha^j_i < m$, otherwise $\{B_i \partial_{x_i}^\ast \mathbf{c}_m\}_{j}=0$. Assume that $\alpha^j_i < m$ and let $j < l \leq n$ be the index such that $\alpha^j+\mathbf{e}_i = \alpha^l$. Then the  above reads $\{B_i\}_{jl} =  \sqrt{\alpha_i^l}$. Thus $B_i$ is upper triangular with zeros on the diagonal. Similarly, it's easy to check $\{B_i^T \partial_{x_i} \B{c}_m\}_j  = \sqrt{\alpha^j_i}\partial_{x_i} c_m^{\alpha^j-\B{e}_i}$,
  where $B_i^T$ is the transpose of $B_i$.
\end{proof}

Using that $\partial_{x_i}^\ast$ is the adjoint of $\partial_{x_i}$ in $\L^2_\eta(\R^d)$ we see that
\begin{equation*}
  (B_i \partial_{x_i}^\ast \mathbf{c}_m, \mathbf{c}_m )_{\L^2_\eta} = (B_i \mathbf{c}_m, \partial_{x_i} \mathbf{c}_m )_{\L^2_\eta} = (B_i^T \partial_{x_i} \mathbf{c}_m, \mathbf{c}_m)_{\L^2_\eta}.
\end{equation*}
Therefore
\begin{align}\label{eq:skew}
  \bigl ( A \mathbf{c}_m  - \sum_{i=1}^d B_i \partial_{x_i}^\ast \mathbf{c}_m + \sum_{i=1}^d B_i^T \partial_{x_i} \mathbf{c}_m, \mathbf{c}_m \bigr )_{\L^2_\eta} = \bigl ( A \mathbf{c}_m, \mathbf{c}_m \bigr )_{\L^2_\eta}.
\end{align}

\begin{remark}
  Note that the matrix $A$ in \cref{lem:skew} appears from the representation of the Ornstein-Uhlenbeck operator \cref{eq:ornuhl} in the spectral expansion, while $B_i$ and $B_i^T$ correspond to $\nabla_v$ and $\nabla_v^*$.
  It can therefore be seen, see \cref{lem:Em} below, that   $B_i\partial_{x_i}^\ast \B{c}_m$ and $B_i^T \partial_{x_i} \B{c}_m$ are mixed second order terms, with one derivative with respect to $v$ and one derivative with respect to  $x_i$.
\end{remark}

\begin{remark}
  In the case $d=1$ the system in \cref{eq:hyperbolic_intro+} equals
  \begin{align*}
    \begin{array}{lllll}
      \partial_t c_m^0 &  & - \sqrt{1} \partial_x^* c_m^1 & & = f^0, \\
      \partial_t c_m^1 & +  c_m^1 & - \sqrt{2} \partial_x^* c_m^2 & + \sqrt{1} \partial_x c_m^0 & = f^1,\\
      \vdots & \vdots & \vdots & \vdots & \vdots \\
      \partial_t c_m^{m-1} & +  (m-1) c_m^{m-1} & - \sqrt{m}\partial_x^* c_m^m & +  \sqrt{m-1} \partial_x c_m^{m-2} & = f^{m-1},\\
      \partial_t c_m^m & +  m c_m^m &  & +  \sqrt{m} \partial_x c_m^{m-1} & = f^m,
    \end{array}
  \end{align*}
  and the matrices $A$ and $B$ appearing in \cref{lem:skew} are given as
  \begin{align*}
    A = \begin{pmatrix}
    0 &  &  &  &  \\
    & {1} &  &  &  \\
    &  & {2} &  &  \\
    &  &  & \ddots &  \\
    &  &  &  & {m}
    \end{pmatrix},
    \quad
    B = \begin{pmatrix}
    0 & \sqrt{1} & 0 &  &  \\
    & 0 & \sqrt{2} &  &  \\
    &  & 0 & \ddots &  \\
    &  &  & \ddots & \sqrt{m} \\
    &  &  &  & 0
    \end{pmatrix}.
  \end{align*}
\end{remark}

\subsection{Well-posedness for the hyperbolic system}
From \cref{lem:skew} we see that if we define
\begin{equation*}
  H  :=  A  -  \sum_{i=1}^d B_i \partial_{x_i}^\ast  + \sum_{i=1}^d B_i^T \partial_{x_i},
\end{equation*}
then the initial value problem corresponding to \cref{eq:hyperbolic_intro+} is
\begin{align}\label{ivp2}
  \begin{cases}
    \partial_t \mathbf{c}_m + H \mathbf{c}_m =  \B{f}_m & \text{ in } (0,T)\times\T^d,\\
    \hphantom{\partial_t \mathbf{c}_m + H } \mathbf{c}_m = \mathbf{g}_m & \text{ on } \{t=0\}\times\T^d,
  \end{cases}
\end{align}
where
\begin{equation*}
  \mbox{$\mathbf{g}_m = (g_m^1,\dots, g_m^n)$ with $g_m^j = (g,\Psi_{\alpha^j})_{\L^2_\mu}$,}
\end{equation*}
and
\begin{equation*}
  \mbox{$\B{f}_m = (f^{\alpha^1}, \dots,f^{\alpha^n})$ with $f^{\alpha^j} = (f, \Psi_{\alpha^j})_{\L^2_\mu}$ for all $1\leq j\leq n$}.
\end{equation*}

\begin{definition} \label{def:brinkman_weak} 
  We say that a vector-valued function
  $\B{c}_m \in \L^2((0,T);\mathrm{H}^1_\eta(\T^d;\R^n))$ is a weak solution to \cref{ivp2} if $\partial_t \B{c}_m \in \L^2((0,T);\mathrm{H}^1_\eta(\T^d;\R^n))$,
  $\mathbf{c}_m(0) = \mathbf{g}_m$ and for a.e. $0 \leq t \leq T$ and all test functions
  $\varphi \in \mathrm{H}^1_\eta(\T^d;\R^n)$, it holds
  \begin{equation*}
    (\partial_t \mathbf{c}_m(t,\cdot) + H \mathbf{c}_m(t,\cdot), \varphi(\cdot))_{\L^2_\eta} = 0.
  \end{equation*}
\end{definition}

To prove existence of weak solutions to \cref{ivp2} we will use the method of vanishing viscosity, i.e.~we consider, for $\varepsilon>0$ small, the problem
\begin{align}\label{ivp3}
  \begin{cases}
    \partial_t \mathbf{c}^\varepsilon_m + \varepsilon \nabla_x^\ast \nabla_x \mathbf{c}_m^\varepsilon + H \mathbf{c}_m^\varepsilon= \B{f}^\varepsilon_m &\text{ in }  (0,T)\times\T^d, \\
    \hphantom{\partial_t \mathbf{c}^\varepsilon_m + \varepsilon \nabla_x^\ast \nabla_x \mathbf{c}_m^\varepsilon + H}\mathbf{c}^\varepsilon_m = \mathbf{g}^\varepsilon_m & \text{ on } \{t=0\}\times\T^d,
  \end{cases}
\end{align}
where
\begin{equation*}
  \mbox{$\mathbf{g}^\varepsilon_m := \phi_\varepsilon * \mathbf{g}_m$ and $\B{f}_m^\varepsilon := \phi_\varepsilon * \B{f}_m$},
\end{equation*}
where $\phi_\varepsilon(x) = \varepsilon^{-n} \phi(x/\varepsilon)$, and $\phi$ is a standard mollifier on $\T^d$. Note that the convolution is applied componentwise on the vector valued functions.
\begin{proposition}\label{thm:hsys:approx}
  Let $k\geq 0$ be a fixed integer and assume that
  \begin{equation}\label{eq:assumphsys}
    U(x)\in C^\infty(\T^d),\ \B{g}_m\in\mathrm{H}^{k+1}_\eta(\T^d),\ \B{f}_m\in \mathrm{H}^{k+1}_{t,x}((0,T)\times\T^d).
  \end{equation}
  Then there exists, for each $\varepsilon>0$, a unique weak solution $\mathbf{c}_m^\varepsilon$ to \cref{ivp3} such that
  \begin{equation}\label{regga}
    \partial_t^{i} \B{c}_m^\varepsilon \in \L^2((0,T);\mathrm{H}^{2k+4-2i}_\eta(\T^d)), \text{ for every } i=0,\dots,k+2.
  \end{equation}
\end{proposition}
\begin{proof} Note that the measure $\ud \eta = e^{-U(x)} \ud x$ is comparable to the Lebesgue measure on $\T^d$, and therefore we can and will in the following  use function spaces defined with respect to  $\ud x$ instead of $\ud \eta$. Introducing
  \begin{equation*}
    X := \L^\infty (0,T; \mathrm{H}^1(\T^d; \R^{n})),
  \end{equation*}
  we consider for $\mathbf{w}\in X$ given, the linear system
  \begin{align*}
    \begin{cases}
      \partial_t \mathbf{c}_m - \varepsilon \Delta_{x} \mathbf{c}_m = \B{f}_m^\varepsilon- \varepsilon \nabla_x U\cdot\nabla_x\mathbf{w}- H \mathbf{w} & \text{ in } (0,T)\times\T^d, \\
      \hphantom{\partial_t \mathbf{c}_m - \varepsilon \Delta_{x}} \mathbf{c}_m = \B{g}_m^\varepsilon& \text{ on } \{t=0\}\times\T^d.
    \end{cases}
  \end{align*}
  Given our information on the right hand side of this system, and on $\B{g}_m^\varepsilon$, we can conclude that this system has a unique solution $\mathbf{c}_m$ such that $\mathbf{c}_m\in \L^2((0,T);\allowbreak\mathrm{H}^2(\T^d;\allowbreak\R^n))$, $\partial_t \mathbf{c}_m\in \L^2((0,T);\allowbreak\L^2(\T^d;\allowbreak\R^n))$. Therefore, it follows from a standard fixed point argument, see \cite[Sec 7.3.2.b]{evansPartialDifferentialEquations2010a},  that there exists a unique  solution $\mathbf{c}_m$ to \cref{ivp3} satisfying $\mathbf{c}_m\in \L^2((0,T); \mathrm{H}^1(\T^d;\R^n))$ and $\partial_t \mathbf{c}_m \in \L^2((0,T);\mathrm H^{-1}(\T^d;\R^n))$.
  The regularity result stated in \cref{regga} is a consequence of  parabolic regularity theory as we have smooth initial data and regular coefficients, see for instance \cite[Sec 7.1.3]{evansPartialDifferentialEquations2010a}.
\end{proof}

Having established the existence of solutions to \cref{ivp3} we next prove uniform energy estimates w.r.t.~$\varepsilon$ in order to complete the method of vanishing viscosity.

\begin{proposition}\label{thm:energy}
  Let $k\geq 0$ be a fixed integer and assume \cref{eq:assumphsys}. Let $\mathbf{c}_m^\varepsilon$ be the unique weak solution to \cref{ivp3} obtained in \cref{thm:hsys:approx}. Given an integer $k\geq 0$, there exists a constant $C = C(d,T,m,U,k)>0$ such that
  \begin{equation}\label{thm:energy:estimate}
    \norm{\B{c}_m^\varepsilon}^2_{\mathrm H^{k+1}_{t,x}((0,T)\times \T^d)} \leq C \bigl(\norm{\B{g}_m}^2_{\mathrm H^{k+1}_\eta(\T^d)} + \norm{\B{f}_m}^2_{\mathrm H^{k+1}_{t,x}((0,T)\times \T^d)}\bigr).
  \end{equation}
\end{proposition}

\begin{proof}
  We will prove the proposition by induction with respect to $k$. Throughout this proof we use the convention that $a \lesssim b$ means that $a\leq Cb$ for some positive constant $C$ which may depend on $(d,T,m,U,k)$ but which is independent of $\varepsilon$. In the following we will use that the commutator $[\partial_{x_j}, \partial_{x_i}^\ast]$ equals multiplication by $\partial_{x_j}\partial_{x_i} U$.

  \smallskip

  \noindent \emph{Step 1: }
  In this step we establish \cref{thm:energy:estimate} in the case $k=0$ assuming $\B{g}_m\in \mathrm H^1_\eta(\T^d)$ and $\B{f}_m\in \mathrm H^1_{t,x}((0,T)\times\T^d)$.  By \cref{eq:skew} and the self-adjointness of $\nabla_x^*\nabla_x$ in $\L^2_\eta(\T^d)$,
  \begin{align}\label{eq:hsys:main}
    \begin{split}
      \partial_t \frac{1}{2} \norm{\B{c}_m^\varepsilon}^2_{\L^2_\eta} =& - \norm{A^{1/2}\B{c}_{m}^\varepsilon}^2_{\L^2_\eta} - \varepsilon \norm{\nabla_x \B{c}_m^\varepsilon}^2_{\L^2_\eta} + (\B{c}_m^\varepsilon, \B{f}_m^\varepsilon)_{\L^2_\eta} \\
      \leq & (\B{c}_m^\varepsilon, \B{f}_m^\varepsilon)_{\L^2_\eta} \leq \norm{\B{c}_m^\varepsilon}^2_{\L^2_\eta} + \norm{\B{f}_m^\varepsilon}^2_{\L^2_\eta}.
    \end{split}
  \end{align}
  Applying Grönwall's inequality we get
  \begin{equation}\label{ineq:hsys:main}
    \max_{0\leq t\leq T} \norm{\B{c}_m^\varepsilon}^2_{\L^2_\eta} \lesssim  \norm{\B{g}_m^\varepsilon}^2_{\L^2_\eta} + \norm{\B{f}_m^\varepsilon}^2_{\L^2((0,T);\L^2_\eta)}.
  \end{equation}
  Using \cref{thm:hsys:approx} we see that we can now differentiate \cref{ivp3} w.r.t.~$x$. Given a multi-index $\theta = (\theta_1, \dots,\theta_d)$, we set $\nabla_x^\theta := \partial_{x_{1}}^{\theta_1}\cdots\partial_{x_{d}}^{\theta_d}$,  and we let $\B{w}^{(\theta)} :=\nabla_x^\theta \B{c}_m^\varepsilon$ and similarly for $(\B{g}_m^\varepsilon)^{(\theta)}$. We also let
  \begin{equation*}
    \norm{\nabla_x^l \B{c}_m^\varepsilon}^2_{\L^2_\eta}:= \sum_{\abs{\theta}=l} \norm{\B{w}^{(\theta)}}^2_{\L^2_\eta},
  \end{equation*}
  for all integers $l\geq 0$. Using this notation we have, in particular, that $\B{w}^{(\B{e}_j)}=\partial_{x_j}\B{c}_m^\varepsilon$, and
  differentiating the equation in \cref{ivp3} w.r.t.~$x_j$, for $1\leq j\leq d$, we deduce
  \begin{align} \label{eq:hsys:dx}
    \begin{cases}
      \partial_t \B{w}^{(\B{e}_j)} + H \B{w}^{(\B{e}_j)}  + \varepsilon \nabla_{x}^\ast \nabla_{x} \B{w}^{(\B{e}_j)}
      = \tilde{\B{f}}^{(\B{e}_j)} & \text{ in } (0,T)\times \T^d ,\\
      \hphantom{\partial_t \B{w}^{(\B{e}_j)} + H \B{w}^{(\B{e}_j)}  + \varepsilon \nabla_{x}^\ast \nabla_{x}} \B{w}^{(\B{e}_j)} = \partial_{x_j} \B{g}_m^\varepsilon &\text{ on } \{t=0\}\times\T^d.
    \end{cases}
  \end{align}
  Here
  \begin{equation*}
    \tilde{\B{f}}^{(\B{e}_j)} = \sum_{i=1}^d [\partial_{x_j}, \partial_{x_i}^\ast] \bigl (B_i \B{c}_m^\varepsilon - \varepsilon \B{w}^{(\B{e}_i)} \bigr ) + \partial_{x_j} \B{f}_m^\varepsilon,
  \end{equation*}
  and it follows that
  \begin{equation*}
    \norm{{\tilde{\B{f}}}^{(\B{e}_j)}}_{\L^2_\eta}^2 \lesssim \norm{\B{c}_m^\varepsilon}_{\L^2_\eta}^2 + \norm{\nabla_x \B{c}_m^\varepsilon}_{\L^2_\eta}^2 + \norm{\partial_{x_j}\B{f}_m^\varepsilon}_{\L^2_\eta}^2.
  \end{equation*}
  Arguing again as in \cref{eq:hsys:main} and applying Gr\"onwall's inequality we obtain
  \begin{equation}\label{ineq:hsys:dx}
    \max_{0\leq t\leq T} \norm{\nabla_x \B{c}_m^\varepsilon}^2_{\L^2_\eta} \lesssim \norm{\B{g}_m^\varepsilon}^2_{\mathrm{H}^1_\eta} + \norm{\B{f}_m^\varepsilon}^2_{\L^2((0,T);\mathrm{H}^1_\eta)}.
  \end{equation}
  Next, letting $\hat{\B{w}} = \partial_t \B{c}_m^\varepsilon$ we see that
  \begin{align}\label{eq:hsys:dt}
    \begin{cases}
      \partial_t \hat{\B{w}} + H \hat{\B{w}}  + \varepsilon \nabla_{x}^\ast \nabla_{x} \hat{\B{w}}
      = \partial_t \B{f}_m^\varepsilon & \text{ in } (0,T)\times \T^d ,\\
      \hat{\B{w}} = \B{f}_m^\varepsilon(0) - H \B{g}_m^\varepsilon - \varepsilon\nabla_x^\ast \nabla_x \B{g}_m^\varepsilon & \text{ on } \{t=0\}\times\T^d.
    \end{cases}
  \end{align}
  Arguing again as in \cref{eq:hsys:main}, we obtain
  \begin{align}\label{ineq:hsys:dt:0}
    \begin{split}
      \norm{\hat{\B{w}}}^2_{\L^2_\eta} &\lesssim \norm{\hat{\B{w}}(0)}^2_{\L^2_\eta} + \norm{\partial_t \B{f}_m^\varepsilon}^2_{\L^2((0,T);\L^2_\eta)}\\
      & \lesssim \norm{\B{f}_m^\varepsilon(0)}^2_{\L^2_\eta} +\norm{ H \B{g}_m^\varepsilon}^2_{\L^2_\eta} + \varepsilon^2 \norm{\nabla_x^\ast \nabla_x \B{g}_m^\varepsilon}^2_{\L^2_\eta} + \norm{\partial_t \B{f}_m^\varepsilon}^2_{\L^2((0,T);\L^2_\eta)}.
    \end{split}
  \end{align}
  Since  $\B{g}_m^\varepsilon = \phi_\varepsilon * \B{g}_m$ and $\T^d$ is bounded, it holds
  \begin{equation}\label{ineq:hsys:dt:1}
    \norm{\nabla_x^\ast \nabla_x \B{g}_m^\varepsilon}^2_{\L^2_\eta} \leq \frac{C}{\varepsilon^2} \norm{ \nabla_x \B{g}_m}^2_{\L^2_\eta}.
  \end{equation}
  Also, by the fundamental theorem of calculus and Jensen's inequality, we have
  \begin{equation}\label{ineq:hsys:dt:2}
    \norm{\B{f}_m^\varepsilon(0)}^2_{\L^2_\eta} \lesssim \norm{\B{f}_m^\varepsilon}^2_{\L^2((0,T);\L^2_\eta)} + \norm{\partial_t \B{f}_m^\varepsilon}^2_{\L^2((0,T);\L^2_\eta)}.
  \end{equation}
  Furthermore, by the definition of $H$,
  \begin{equation}\label{ineq:hsys:dt:3}
    \norm{ H \B{g}_m^\varepsilon}^2_{\L^2_\eta} \lesssim \norm{\B{g}_m^\varepsilon}^2_{\mathrm{H}^1_\eta}.
  \end{equation}
  Combining \cref{ineq:hsys:dt:0,ineq:hsys:dt:1,ineq:hsys:dt:2,ineq:hsys:dt:3} we  see that
  \begin{equation}\label{ineq:hsys:dt}
    \norm{\partial_t \B{c}_m^\varepsilon} ^2_{\L^2_\eta} \lesssim \norm{\B{g}_m^\varepsilon}^2_{\mathrm{H}^1_\eta} + \norm{\B{f}_m^\varepsilon}^2_{\L^2((0,T);\L^2_\eta)} + \norm{\partial_t \B{f}_m^\varepsilon}^2_{\L^2((0,T);\L^2_\eta)}.
  \end{equation}
  Finally, we note that
  \begin{equation*}
    \norm{\B{g}_m^\varepsilon}^2_{\mathrm{H}^1_\eta} \leq \norm{\B{g}_m}^2_{\mathrm{H}^1_\eta} \textrm{ and } \norm{\B{f}_m^\varepsilon}^2_{\mathrm{H}^1_{t,x}} \leq \norm{\B{f}_m}^2_{\mathrm{H}^1_{t,x}}.
  \end{equation*}
  These two inequalities, combined with \cref{ineq:hsys:main,ineq:hsys:dx,ineq:hsys:dt}, completes the proof of \cref{thm:energy:estimate} in the case $k=0$.

  \smallskip
  \noindent \emph{Step 2: }
  In this step we assume that \cref{thm:energy:estimate} holds for some $k\geq 0$ and we want to prove
  the estimate for $k+1$ assuming that $\B{g}_m\in \mathrm{H}^{k+2}_{\eta}(\T^d)$ and $\B{f}_m\in\mathrm{H}^{k+2}_{t,x}((0,T)\times \T^d)$. As $\mathrm{H}^{k+2}_{\eta}(\T^d)\subset \mathrm{H}^{k+1}_{\eta}(\T^d)$ and $\mathrm{H}^{k+2}_{t,x}((0,T)\times \T^d)\subset\mathrm{H}^{k+1}_{t,x}((0,T)\times \T^d)$ it follows by the induction hypothesis that
  \begin{equation}\label{ineq:hsys:induction:main}
    \norm{\B{c}_m^\varepsilon}^2_{\mathrm H^{k+1}_{t,x}((0,T)\times \T^d)} \leq C \bigl(\norm{\B{g}_m}^2_{\mathrm H^{k+1}_\eta} + \norm{\B{f}_m}^2_{\mathrm H^{k+1}_{t,x}((0,T)\times \T^d)}\bigr).
  \end{equation}
  In a  way similar to \cref{eq:hsys:dt}, just noting that $\partial_t \B{f}^\varepsilon_m$ and $\hat{\B{w}}(0) $ now satisfy the assumption \cref{eq:assumphsys} for $k$, it follows that
  \begin{equation*}
    \sum_{i+j=0}^{k+1} \norm{{\partial_t^i \nabla_x^j \hat{\mathbf{w}}}}^2_{\L^2((0,T);\L^2_\eta)} \lesssim \norm{\hat{\B{w}}(0)}^2_{\mathrm{H}^{k+1}_\eta}+ \sum_{i+j=0}^{k+1} \norm{\partial_t^{i+1} \nabla_x^{j} \B{f}_m}^2_{\L^2((0,T);\L^2_\eta)},
  \end{equation*}
  i.e.
  \begin{align}\label{ineq:hsys:induction:dt}
    \sum_{i=1}^{k+2} \norm{\partial_t^i \B{c}_m^\varepsilon}^2_{\L^2((0,T);\mathrm{H}^{k+2-i}_\eta)} \lesssim \norm{\B{g}_m}^2_{\mathrm{H}^{k+2}_\eta}
    + \sum_{i=0}^{k+2} \norm{\partial_t^i \B{f}_m}^2_{\L^2((0,T);\mathrm{H}^{k+2-i}_\eta)}.
  \end{align}
  It remains to control $\norm{\nabla_x^{k+2} \B{c}_m^\varepsilon}_{\L^2_\eta}$. Recall the notations in \cref{eq:hsys:dx} and assume that $\theta$ is a multi-index such that $\abs{\theta} = k+1$. Then, for any $\theta+\B{e}_j$ with $1\leq j\leq d$ we have
  \begin{align*}
    \begin{cases}
      \partial_t \B{w}^{(\theta+\B{e}_j)}+ H \B{w}^{(\theta+\B{e}_j)} + \varepsilon \nabla_{x}^\ast \nabla_{x} \B{w}^{(\theta+\B{e}_j)}
      = \tilde{\B{f}}^{(\theta+\B{e}_j)} & \text{ in } (0,T)\times \T^d,\\
      \hphantom{\partial_t \B{w}^{(\theta+\B{e}_j)}+ H \B{w}^{(\theta+\B{e}_j)} + \varepsilon \nabla_{x}^\ast \nabla_{x}}  \B{w}^{(\theta+\B{e}_j)} = {(\B{g}_m^\varepsilon)}^{(\theta+\B{e}_j)} & \text{ on } \{t=0\}\times\T^d,
    \end{cases}
  \end{align*}
  where the functions $\{\tilde{\B{f}}^{(\theta+\B{e}_j)}\}$ are defined inductively as
  \begin{equation*}
    \tilde{\B{f}}^{(\theta+\B{e}_j)}= \sum_{i=1}^d [\partial_{x_j}, \partial_{x_i}^*] \bigl( B_i \B{w}^{(\theta)} - \varepsilon \B{w}^{(\theta+\B{e}_i)} \bigr) + \partial_{x_j} \tilde{\B{f}}^{(\theta)},
  \end{equation*}
  with $\tilde{\B{f}}^{(0)} = \B{f}_m^\varepsilon$. Proceeding similarly to the deduction in \cref{ineq:hsys:dx} we can conclude
  \begin{equation*}
    \norm{\nabla_x^{k+2} \B{c}_m^\varepsilon}^2_{\L^2_\eta} \lesssim \norm{\B{g}_m}^2_{\mathrm{H}^{k+2}_\eta(\T^d)} + \norm{\B{f}_m}^2_{\L^2((0,T); \mathrm{H}^{k+2}_\eta)((0,T)\times \T^d)}.
  \end{equation*}
  This inequality, together with \cref{ineq:hsys:induction:main,ineq:hsys:induction:dt}, completes the induction and hence the proof of the proposition.
\end{proof}

\begin{proposition} \label{prop:existence}
  Let $k\geq 0$ be a fixed integer and assume \cref{eq:assumphsys}.  Then there exists a unique weak solution in the sense of \cref{def:brinkman_weak}, $\mathbf{c}_m \in \mathrm{H}^{k+1}_{t,x}((0,T)\times \T^d)$, to the initial value problem \cref{ivp2}. Furthermore, there exists a constant $C = C(d,T,m,U,k)>0$ such that
  \begin{align*}
    \norm{\B{c}_m}^2_{\mathrm H^{k+1}_{t,x}((0,T)\times \T^d)} \leq C \bigl(\norm{\B{g}_m}^2_{\mathrm H^{k+1}_\eta(\T^d)} + \norm{\B{f}_m}^2_{\mathrm H^{k+1}_{t,x}((0,T)\times \T^d)}\bigr).
  \end{align*}
\end{proposition}
\begin{proof}
  Using the uniform in $\varepsilon$ estimates in \cref{thm:energy}, we can argue as in \cite[Chapter 7.3]{evansPartialDifferentialEquations2010a} to obtain the existence and uniqueness of a weak solution. The quantitative estimate follows readily.
\end{proof}

\section{Uniform energy estimates for the Galerkin approximations}\label{sec:uniformestimate}
In this section we prove a uniform energy estimate for the Galerkin approximations $\{u_m\}$ stated in \cref{eq:ansatz} assuming that $\B{c}_m$ solves \cref{ivp2}. Note that by the orthogonality of Hermite functions we have
\begin{equation*}
  \norm{u_m}^2_{\L^2((0,T);\mathrm{H}^k_\eta(\L^2_\mu))} = \sum_{\abs{\alpha}=0}^m \norm{c_m^\alpha}_{\L^2((0,T);\mathrm{H}^k_\eta)}^2.
\end{equation*}
Using this observation and \cref{prop:existence} we can conclude the validity of the following lemma.
\begin{lemma}
  Let $k\geq 0$ be a fixed integer and assume \cref{eq:assumphsys}. Let $u_m$ be of the form \cref{eq:ansatz} with $\B{c}_m$ solving \cref{ivp2}.
  Then
  \begin{equation*}
    \partial_t^i u_m \in\L^2\bigl(0,T; \mathrm{H}^{k+1-i}_\eta(\L^2_\mu) \bigr),\quad \textrm{ for all } i=0, \dots, k+1.
  \end{equation*}
  Furthermore, $v\mapsto \partial_t^i\nabla_x^j u_m(\cdot, \cdot, v)$ is smooth for all  non-negative integers $i$, $j$ such that $i+j\leq k+1$.
\end{lemma}

Thus the derivation \cref{eq:um:expansion_err} is now established, i.e.
\begin{align}\label{eq:Lum}
  (\partial_t - \mathcal{L}) u_m = f_m + E_m,
\end{align}
where
\begin{align}\label{eq:Lum+}
  f_m = \sum_{\abs{\alpha}=0}^m f^\alpha \Psi_\alpha \text{ and } {E}_m = \sum_{\abs{\alpha} = m} \sum_{i=1}^d \sqrt{\alpha_i+1} \partial_{x_i} c_m^\alpha \Psi_{\alpha+\mathbf{e}_i} .
\end{align}

In order to prove \cref{Thm1} we need, as explained in \cref{sec:outline}, to provide regularity estimates for $u_m$ as well as estimates for the decay of the error $E_m$. This is the content of \cref{lem:Em,thm:regularity}.

\begin{lemma} \label{lem:Em}
  \hphantom{1em}
  \begin{enumerate}
    \item Let $u_m$ be of the form \cref{eq:ansatz}. Then
    \begin{equation*}
      \norm{v\cdot \nabla_x u_m(t,\cdot,\cdot)}_{\L^2_\eta(\L^2_\mu)} \leq C\norm{\nabla_x u_m(t,\cdot,\cdot)}_{\L^2_\eta(\mathrm{H}^1_\mu)},
    \end{equation*}
    for a constant $C$ which depends only on the dimension $d$.
    \item
    Let ${E}_m$ be as defined in \cref{eq:Lum+}. Then for $k \in \N_+$, we have
    \begin{equation}\label{eq:um:Em}
      \norm{{E}_m(t,\cdot,\cdot)}^2_{\L^2_\eta(\L^2_\mu)} \leq \frac{d}{(1+m)^k} \norm{\nabla_x u_m(t,\cdot,\cdot)}^2_{\L^2_\eta(\mathrm{H}^{k+1}_\mu)}.
    \end{equation}
  \end{enumerate}
\end{lemma}
\begin{proof}
  Using the relation $v_i \Psi_\alpha =  (\sqrt{\alpha_i+1}\Psi_{\alpha+\B{e}_i}+\sqrt{\alpha_i}\Psi_{\alpha-\B{e}_i})$, see \cref{eq:Hermiterelations}, we can conclude that there is a constant $C=C(d)$ such that
  \begin{align*}
    \|v \cdot \nabla_x u_m(t,\cdot,\cdot)\|_{\L^2_\eta(\L^2_\mu)}^2
    =&
    \Big \|\sum_{\abs{\alpha}=0}^m \sum_{i=1}^d \partial_{x_i} c_m^\alpha \Big( \sqrt{\alpha_i+1} \Psi_{\alpha+\mathbf{e}_i} + \sqrt{\alpha_i}\Psi_{\alpha-\mathbf{e}_i} \Big ) \Big \|_{\L^2_\eta(\L^2_\mu)}^2 \\
    \leq&
    C \Big \|\sum_{\abs{\alpha}=0}^m \sum_{i=1}^d \partial_{x_i} c_m^\alpha \sqrt{\alpha_i+1} \Psi_{\alpha+\mathbf{e}_i}\Big \|^2_{\L^2_\eta(\L^2_\mu)} \\
    &+
    C\Big \|\sum_{\abs{\alpha}=0}^m \sum_{i=1}^d \partial_{x_i} c_m^\alpha \sqrt{\alpha_i}\Psi_{\alpha-\mathbf{e}_i} \Big \|_{\L^2_\eta(\L^2_\mu)}^2
    \\
    \leq& C \sum_{\abs{\alpha}=0}^m (1+|\alpha|)\| \nabla_{x} c_m^\alpha\|_{\L^2_\eta}^2 \leq C\norm{\nabla_x u_m(t,\cdot,\cdot) }^2_{\L^2_\eta(\mathrm{H}^1_\mu)} .
  \end{align*}
  This proves the first statement in the lemma. To prove the second statement, using the orthogonality of $\Psi_\alpha$ we have
  \begin{equation*}
    \|{E}_m(t,\cdot,\cdot)\|_{\L^2_\eta(\L^2_\mu)}^2 =  \sum_{\abs{\alpha} = m}(d+m) \| \nabla_{x} c_m^\alpha(t,\cdot)\|_{\L^2_\eta}^2 \leq d \sum_{\abs{\alpha}=m} (1+m)\norm{\nabla_x c_m^\alpha(t,\cdot)}^2_{\L^2_\eta}.
  \end{equation*}
  We then observe that by definition \cref{eq:fractional_norm},
  \begin{align*}
    \sum_{\abs{\alpha}=m} (1+m) \norm{\nabla_x c_m^\alpha(t,\cdot)}^2_{\L^2_\eta}
    &=
    \frac{(1+m)^{k}}{(1+m)^k} \sum_{\abs{\alpha}=m} (1+m) \norm{\nabla_x c_m^\alpha(t,\cdot)}^2_{\L^2_\eta}
    \\
    &\leq \frac{1}{(1+m)^k}
    \norm{\nabla_x u_m(t,\cdot,\cdot)}^2_{\L^2_\eta(\mathrm{H}^{k+1}_\mu)}.
  \end{align*}
  Combining the two inequalities above we get the inequality in \cref{eq:um:Em}.
\end{proof}

To control the norms appearing on the right hand side in the inequalities in \cref{lem:Em}, we need to prove regularity in $v$ of $u_m$.
Previously we proved regularity in the $x$ variable using a standard bootstrap technique, we will now continue to prove the regularity in the velocity variable phrased in our spectral space.
Recall that from \cref{eq:fractional_norm} that the regularity in $v$ corresponds to the decay of $\norm{c_m^\alpha}_{\L^2_\eta}$ w.r.t.~$\abs{\alpha}$.
Furthermore, differentiation w.r.t.~$v$ can in the spectral space essentially be written as multiplication with powers of the matrix $A$.
When deriving the energy estimates for $(I+A)^s \B{c}_m$, the cancellation seen in \cref{eq:skew} will be destroyed, specifically the commutator $[B_i,(I+A)^s] \neq 0$ when $s > 0$. However, as we will see in the proof, the remainder term will be of lower order.

\begin{proposition} \label{thm:regularity}
  Let the assumptions in
  \cref{eq:assump:um1}
  hold for a fixed integer $k\geq0$
  and let $u_m$ be of the form \cref{eq:ansatz} with $\B{c}_m$ solving \cref{ivp2} with initial data $\B{g}_m$ and source term $\B{f}_m$.
  Then there exists a constant $C=C(d,T,\norm{\nabla_x^2 U}_\infty, \allowbreak \dots,\allowbreak \norm{\nabla_x^{k+1}U}_\infty,k)>0$, independent of $m$, such that
  \begin{equation*}
    \norm{u_m}^2_{\mathrm{H}^k_{t,x,v}((0,T)\times\T^d\times\R^d)} \leq C\bigl(\norm{g}^2_{\mathrm{H}^k_{x,v}(\T^d\times\R^d)} + \norm{f}^2_{\mathrm{H}^k_{t,x,v}((0,T)\times\T^d\times\R^d)}\bigr).
  \end{equation*}
\end{proposition}
\begin{proof}  Throughout this proof we use the convention that $a \lesssim b$ means that $a\leq Cb$ for some positive constant $C$ which may depend on $d$, $T$, $\norm{\nabla_x^2 U}_\infty, \allowbreak \dots,\allowbreak \norm{\nabla_x^{k+1}U}_\infty$, $k$ but which is independent of $m$.  To start the argument we first observe that
  \begin{align}\label{esta}
    \norm{ u_m }^2_{\mathrm{H}^k_{t,x,v}((0,T)\times \T^d\times\R^d)} =& \sum_{i+j+l = 0}^k \norm{\partial_t^i \nabla_x^j \nabla_v^{2l} u_m}^2_{\L^2((0,T);\L^2_\eta(\T^d;\L^2_\mu(\R^d)))}\notag\\
    =& \sum_{i=0}^k \norm{\partial_t^i u_m}^2_{\L^2((0,T);\mathrm{H}^{k-i}_{x,v}(\T^d\times\R^d))}.
  \end{align}
  Recall that  $g_m = \sum_{\abs{\alpha}=0}^m (g,\Psi_\alpha)_{\L^2_\mu}\Psi_\alpha$ and $f_m=\sum_{\abs{\alpha}=0}^m (f,\Psi_\alpha)_{\L^2_\mu}\Psi_\alpha$. Hence, given $k$ we have that $g_m$ and $f_m$ satisfy
  \cref{eq:assump:um1}.
  We first estimate the term in \cref{esta} which corresponds to
  $i=0$. We claim  that
  \begin{equation}\label{eq:um:claim1}
    \sup_{0\leq t\leq T}\norm{u_m(t,\cdot,\cdot)}^2_{\mathrm{H}^k_{x,v}} \lesssim \bigl(\norm{g_{m}}^2_{\mathrm{H}^k_{x,v}(\T^d\times\R^d)} + \norm{f_{m}}^2_{\mathrm{H}^k_{t,x,v}((0,T)\times\T^d\times\R^d)}\bigr)
  \end{equation}
  For efficiency we introduce, for $s,j$ non-negative integers,
  \begin{equation*}
    N(s,j, t):= \norm{\nabla_x^j u_m(t,\cdot,\cdot)}^2_{\L^2_\eta(\mathrm{H}^{2s}_\mu)} \text{ and } F(s,j,t) := \norm{\nabla_x^j f_{m}(t,\cdot,\cdot)}^2_{\L^2_\eta(\mathrm{H}^{2s}_\mu)}.
  \end{equation*}
  To prove the claim \cref{eq:um:claim1}, it suffices to prove that
  \begin{equation}\label{eq:um:sum}
    \partial_t \sum_{s+j\leq k } N(s,j,t) \lesssim \sum_{s+j\leq k} N(s,j,t) + \sum_{s+j\leq k} F(s,j,t).
  \end{equation}
  Indeed we just need to apply Gr\"onwall's inequality to \cref{eq:um:sum} to conclude \cref{eq:um:claim1}.

  To prove \cref{eq:um:sum} we note that
  \begin{equation*}
    N(s,j,t) = \sum_{\abs{\theta} = j}\norm{(I+A)^{s} \nabla_x^\theta \B{c}_m}^2_{\L^2_\eta},
  \end{equation*}
  a similar statement holds for $F$.
  Hence we need to go back to the equation in \cref{ivp2}. Similar to the proof of \cref{thm:energy}, we let  for a multi-index
  $\theta$,
  \begin{align}\label{eq:um:bj}
    \B{w}^{(\theta)} &= \nabla_{x}^\theta \B{c}_m,\quad \B{g}^{(\theta)} =  \nabla_{x}^{\theta} \B{g}_m,\quad \B{f}^{(\theta)} = \nabla_x^\theta \B{f}_{m},\notag \\
    \B{b}^{(\theta+\B{e}_j)} &= \sum_{i=1}^d [\partial_{x_j}, \partial_{x_i}^\ast] B_i \B{w}^{(\theta)} + \nabla_{x_j} \B{b}^{(\theta)},
  \end{align}
  and $\B{b}^{(0)} = 0$. In addition, for a non-negative integer $j$  we introduce
  \begin{equation*}
    \norm{(I+A)^{s} \B{w}^{(j)}}^2_{\L^2_\eta} := \sum_{\abs{\theta}=j} \norm{(I+A)^{s} \B{w}^{(\theta)}}^2_{\L^2_\eta},
  \end{equation*}
  and similarly for $\norm{(I+A)^{s} \B{b}^{(j)}}^2_{\L^2_\eta}$ and $\norm{(I+A)^{s} \B{f}^{(j)}}^2_{\L^2_\eta}$.
  Then $\B{w}^{(\theta+\B{e}_j)}$ solves the problem
  \begin{align*}
    \begin{cases}
      \partial_t \B{w}^{(\theta+\B{e}_j)} + H \B{w}^{(\theta+\B{e}_j)} = \B{f}^{(\theta+\B{e}_j)}+\B{b}^{(\theta+\B{e}_j)} & \text{in } (0,T)\times\T^d,\\
      \hphantom{\partial_t \B{w}^{(\theta+\B{e}_j)} + H }  \B{w}^{(\theta+\B{e}_j)}  = \B{g}^{(\theta+\B{e}_j)} & \text{on } \{t=0\}\times\T^d.
    \end{cases}
  \end{align*}
  We will use the following lemma.

  \begin{lemma}\label{lem:um:iteration}
    Let $k,u_m,\B{f}_m,\B{g}_m$ be as in \cref{thm:regularity},
    and assume that $s,j$ be nonnegative integers such that  $s\leq k$ and $j\leq k$. Then it holds
    \begin{align*}
      \partial_t \frac{1}{2} \norm{(I+A)^{s} \B{w}^{(j)}}^2_{\L^2_\eta} \lesssim & \norm{(I+A)^s \B{w}^{(j)}}^2_{\L^2_\eta}
      + \mathbb{I}_{[s > 0]}\norm{(I+A)^{s-1} \B{w}^{(j+1)}}^2_{\L^2_\eta}\notag \\
      &+ \norm{(I+A)^{s-1/2} \B{b}^{(j)}}^2_{\L^2_\eta} + \norm{(I+A)^{s} \B{f}^{(j)}}^2_{\L^2_\eta}.
    \end{align*}
  \end{lemma}
  Assuming  \cref{lem:um:iteration} we proceed with the proof of \cref{eq:um:sum} and of \cref{thm:regularity}. First we note that by \cref{eq:um:bj} and the definition of the matrices $B_i$ we have by an iteration that for $s,j \geq 0$, it holds
  \begin{equation}\label{eq:b:bound}
    \norm{(I+A)^{s-1/2} \B{b}^{(j)}} \lesssim \sum_{r=0}^{j-1} N(s,r,t).
  \end{equation}
  Now by \cref{lem:um:iteration,eq:b:bound} we have for $s,j \geq 0$ that
  \begin{align}\label{eq:um:master}
    \partial_t N(s,j,t) \lesssim  N(s,j,t) + \mathbb{I}_{[s>0]}N(s-1,j+1,t)
    + \sum_{r=0}^{j-1} N(s, r,t) + F(s,j,t).
  \end{align}
  Summing the inequalities in \cref{eq:um:master} for $s+j\leq k$  we obtain after some reindexing inequality \cref{eq:um:sum}, hence the claim \cref{eq:um:claim1} is proved.

  Finally, by differentiating w.r.t.~$t$, for all $i\geq 1$, we see that the term $\partial^i_t u_m$ is of the form \cref{eq:ansatz} but with $\B{c}_m$ replaced by $\partial^i_t \B{c}_m$. It's also easy to check that $\partial^i_t \B{c}_m$ satisfies \cref{ivp2} with initial condition $\partial^i_t \B{c}_m(0) = \partial_t^{i-1} \B{f}_m-H\partial_t^{i-1}\B{c}_m(0)$, $\partial_t^0 \B{c}_m(0)=\B{g}_m(0)$, and with source term $\partial_t^i \B{f}_m$.
  Therefore the initial data $\partial_t^i u_m(0)$ and the source term $\partial_t^i f_m$ satsify the assumptions in
  \cref{eq:assump:um1}
  for order $k-i$. Finally, by an induction on $0\leq i\leq k$, and by recalling the inequality
  \begin{equation*}
    \norm{g_m}^2_{\mathrm{H}^k_{x,v}}+\norm{f_m}^2_{\mathrm{H}^k_{t,x,v}}\lesssim \norm{g}^2_{\mathrm{H}^k_{x,v}}+\norm{f}^2_{\mathrm{H}^k_{t,x,v}},
  \end{equation*}
  the proof of the proposition is complete.
\end{proof}

\subsection{Proof of \cref{lem:um:iteration}} 
Using \cref{prop:existence} and a calculation we see that the formula
\begin{align}\label{eq:um:energy}
  \partial_t \frac{1}{2} \norm{(I+A)^{s} \B{w}^{(\theta)}}^2_{\L^2_\eta}
  = \bigl( (I+A)^{2s}\B{w}^{(\theta)}, -H \B{w}^{(\theta)} + \B{f}^{(\theta)}+\B{b}^{(\theta)}  \bigr)_{\L^2_\eta},
\end{align}
is valid for any $s\geq0$ and $\abs{\theta} = j\geq 0$. We claim that for $\delta,\varepsilon\in(0,1)$ there exists $C_1=C_1(s)$ such that $C_1(0) = 0$ and
\begin{align}\label{eq:um:energyestimate}
  \partial_t \frac{1}{2} \norm{(I+A)^{s} \B{w}^{(\theta)}}^2_{\L^2_\eta} \leq& -\norm{A^{1/2}(I+A)^s \B{w}^{(\theta)}}^2_{\L^2_\eta}\notag \\
  &+ \bigl( C_1 \delta + \varepsilon \bigr)\norm{(I+A)^{s+1/2} \B{w}^{(\theta)}}^2_{\L^2_\eta}+ \norm{(I+A)^{s} \B{w}^{(\theta)}}^2_{\L^2_\eta}\notag\\
  &+\frac{C_1 }{\delta} \sum_{i=1}^d \norm{(I+A)^{s-1} \B{w}^{(\theta+\B{e}_i)}}^2_{\L^2_\eta}\notag \\
  &+ \frac{1}{\varepsilon} \norm{ (I+A)^{s-1/2} \B{b}^{(\theta)}}^2_{\L^2_\eta} + \norm{ (I+A)^{s} \B{f}^{(\theta)}}^2_{\L^2_\eta}.
\end{align}
Furthermore, we claim that if we choose $\delta=\delta(s)$ and $\varepsilon=\varepsilon(s)$ small enough and so that $C_1\delta + \varepsilon < 1/2$, then
\begin{align}\label{eq:um:claim}
  &-\norm{A^{1/2}(I+A)^s \B{w}^{(\theta)}}^2_{\L^2_\eta} + \bigl( C_1 \delta + \varepsilon \bigr)\norm{(I+A)^{s+1/2} \B{w}^{(\theta)}}^2_{\L^2_\eta}\notag\\
  &\leq \frac{1}{2} \norm{(I+A)^s \B{w}^{(\theta)}}^2_{\L^2_\eta}.
\end{align}
Indeed, by reformulating the above inequality using multi-indices we  find that the coefficient of $\norm{(w^{(\theta)})^\alpha}^2_{\L^2_\eta}$ in
 \cref{eq:um:claim} equals
\begin{equation*}
  - \abs{\alpha} (1+\abs{\alpha})^{2s} + (C_2 \delta  + \varepsilon) (1+\abs{\alpha})^{2s+1}.
\end{equation*}
Letting $l:=\abs{\alpha}$, we have
\begin{equation*}
  - l(1+l)^{2s} + \frac{1}{2} (1+l)^{2s+1} \leq \frac{1}{2} (1+l)^{2s}
\end{equation*}
which proves \cref{eq:um:claim}. The lemma is then proved by substituting \cref{eq:um:claim} into \cref{eq:um:energyestimate} and applying Grönwall's inequality. The claim in \cref{eq:um:energyestimate} remains to be proven.

To prove the claim \cref{eq:um:energyestimate} we note that
\begin{equation*}
  \bigl( (I+A)^{2s}\B{w}^{(\theta)}, - H \B{w}^{(\theta)} \bigr)_{\L^2_\eta}=-\text{I}+\text{II},
\end{equation*}
where
\begin{align*}
  \text{I}:=\sum_{i=1}^d
  \norm{A^{1/2}(I+A)^s \B{w}^{(\theta)}}^2_{\L^2_\eta},\qquad
  \text{II}:=\sum_{i=1}^d \bigl( \partial_{x_i} \B{w}^{(\theta)}, M_i \B{w}^{(\theta)} \bigr)_{\L^2_\eta}
\end{align*}
and $M_i:= \bigl[ (I+A)^{2s} , B_i \bigr]$.
Furthermore, letting
\begin{align*}
  \text{III} := \bigl((I+A)^{2s} \B{w}^{(\theta)}, \B{b}^{(\theta)} \bigr)_{\L^2_\eta},\qquad
  \text{IV} := \bigl((I+A)^{2s} \B{w}^{(\theta)}, \B{f}^{(\theta)} \bigr)_{\L^2_\eta},
\end{align*}
we see that we can rewrite  \cref{eq:um:energy} as
\begin{equation*}
  \partial_t \frac{1}{2} \norm{(I+A)^{s} \B{w}^{(\theta)}}^2_{\L^2_\eta} = -\text{I}+\text{II}+\text{III}+\text{IV}.
\end{equation*}
We need to estimate the terms II, III and IV. To estimate the term II we see, by expanding it in terms of multi-indices, that
\begin{align}\label{eq:um:II}
  \text{II} =  \sum_{i=1}^d \sum_{\abs{\alpha}=0}^m \sqrt{\alpha_i} \bigl( (1+\abs{\alpha})^{2s} - \abs{\alpha}^{2s}\bigr) \bigl( w^{(\theta)}_\alpha , \partial_{x_i} w^{(\theta)}_\alpha \bigr)_{\L^2_\eta}.
\end{align}
Observe that the coefficient in \cref{eq:um:II}, $\sqrt{\alpha_i} \bigl( (1+\abs{\alpha})^{2s} - \abs{\alpha}^{2s}\bigr)$, is of order $\abs{\alpha}^{2s-\frac12} = \abs{\alpha}^{s-1}\abs{\alpha}^{s+\frac12}$. We claim that
\begin{align}\label{ineq:um:ii}
  \text{II} \leq \frac{C_1}{\delta} \sum_{i=1}^d \norm{(I+A)^{s-1} \B{w}^{(\theta+\B{e}_i)}}^2_{\L^2_\eta} +  C_1 \delta  \norm{(I+A)^{s+\frac12} \B{w}^{(\theta)}}^{2}_{\L^2_\eta},
\end{align}
where $C_1 = C_1(s)$, $C_1(0) = 0$ and $\delta > 0$ is a degree of freedom. Indeed, for all $l\in \N$ we have
\begin{equation*}
  \abs{(1+l)^{2s}-l^{2s}} \leq C_1 (1+l)^{2s-1},
\end{equation*}
for a constant $C_1 = C_1(s)$, satisfying $C_1(0) = 0$. Applying this to \cref{eq:um:II} we have for $i=1,\dots,d$ that
\begin{equation*}
  \abs{\sqrt{\alpha_i} \bigl( (1+\abs{\alpha})^{2s} - \abs{\alpha}^{2s}\bigr)}\leq \sqrt{1+\abs{\alpha}}\bigl( (1+\abs{\alpha})^{2s} - \abs{\alpha}^{2s}\bigr) \leq C_1 (1+\abs{\alpha})^{2s-\frac12}.
\end{equation*}
Note that $2s-\frac12 = (s-1) + (s+ \frac12)$, by Young's inequality with $\delta > 0$ we have
\begin{equation*}
  \text{II} \leq \frac{C_1}{\delta} \sum_{i=1}^d \sum_{\abs{\alpha}=0}^m \norm{(1+\abs{\alpha})^{s-1} \partial_{x_i} w_\alpha^{(\theta)}}^2_{\L^2_\eta} + C_1 \delta \sum_{i=1}^d \sum_{\abs{\alpha}=0}^m  \norm{(1+\abs{\alpha})^{s+\frac12} w_\alpha^{(\theta)}}^2_{\L^2_\eta},
\end{equation*}
which proves \cref{ineq:um:ii}. To estimate the term III we proceed analogously, observing that $(I+A)^{2s} = (I+A)^{s+\frac12}(I+A)^{s-\frac12}$, and using Young's inequality, we obtain, for any $\varepsilon > 0$,
\begin{align}\label{ineq:um:iii}
  \text{III} \leq \varepsilon \norm{(I+A)^{s+\frac12} \B{w}^{(\theta)}}^2_{\L^2_\eta} + \frac{1}{\varepsilon} \norm{ (I+A)^{s-\frac12} \B{b}^{(\theta)}}^2_{\L^2_\eta}.
\end{align}
For term IV we have
\begin{align}\label{ineq:um:iv}
  \text{IV} \leq \norm{(I+A)^{s} \B{w}^{(\theta)}}^2_{\L^2_\eta} +  \norm{ (I+A)^{s} \B{f}^{(\theta)}}^2_{\L^2_\eta}.
\end{align}
Combining \cref{ineq:um:ii,ineq:um:iii,ineq:um:iv} we get the claim in \cref{eq:um:energyestimate}.

\section{Proof of \cref{Thm1} and \cref{cor:l2initial}}\label{sec:bddresult}
To start the proof of  \cref{Thm1} we first recall from \cref{eq:Lum} that $u_m$ satisfies
\begin{equation*}
  (\partial_t - \mathcal{L}) u_m = f_m+{E}_m.
\end{equation*}
Furthermore, using \cref{thm:regularity} we see there is a subsequence, still denoted by $\{u_m\}$, that converges weakly to a function $u\in\mathrm{H}^{k}_{t,x,v}((0,T)\times\T^d\times\R^d)\subset\mathrm{H}^1_{\mathrm{kin}}((0,T)\times\T^d\times\R^d)$.
Using all this information, it follows by standard arguments, see \cite[Sec 7.2.2.c]{evansPartialDifferentialEquations2010a} for instance, that $u$ is a weak solution to the equation $(\partial_t -\mathcal{L})u = f$,  $u(0, x, v) = g(x,v)$ in the sense of \cref{def:weak}. The regularity estimates claimed in \cref{Thm1} follow easily from \cref{thm:regularity}. By linearity of the equation, these estimates also imply the uniqueness of $u$.

To prove the error estimates in \cref{Thm1} we introduce
\begin{equation*}
  \hat{u}_m := \sum_{\abs{\alpha}=0}^m \hat{c}_m^\alpha \Psi_\alpha \textrm{ with } \hat{c}_m^\alpha = (u, \Psi_\alpha)_{\L^2_\mu},
\end{equation*}
and we decompose the approximation error as
\begin{equation*}
  u - u_m= (u - \hat{u}_m) + (\hat{u}_m - u_m),
\end{equation*}
where $u_m$ is our approximating solution \cref{eq:ansatz} solving \cref{ivp2}. Defining $\hat{\mathbf{c}}_m:=(\hat{c}_m^{\alpha^1}, \dots,\allowbreak \hat{c}_m^{\alpha^{n(m)}})$, and substituting $\hat{\mathbf{c}}_m -\mathbf{c}_m$ into the equation \cref{ivp2},  we get
\begin{equation*}
  \partial_t (\hat{\mathbf{c}}_m -\mathbf{c}_m) + H (\hat{\mathbf{c}}_m -\mathbf{c}_m) = \hat{\mathbf{E}}_m,
\end{equation*}
where $\hat{\mathbf{E}}_m = (0, \dots,0, \sum_{\abs{\alpha}=m}\sum_{i=1}^{d}\sqrt{\alpha_i +1} \partial_{x_i}^* \hat{c}_m^{\alpha+\mathbf{e}_i})$.
Defining
\begin{equation*}
  \hat{E}_m = \sum_{\abs{\alpha}=m}\sum_{i=1}^{d}\sqrt{\alpha_i +1} \partial_{x_i}^* \hat{c}_m^{\alpha+\mathbf{e}_i} \Psi_{\alpha},
\end{equation*}
it then follows from \cref{thm:regularity}, and the fact that $\hat{u}_m(0)-u_m(0) = 0$, that 
\begin{equation*}
  \norm{\hat{u}_m - u_m}^2_{\mathrm{H}^l_{t,x,v}}\leq C \norm{\hat{E}_m}^2_{\mathrm{H}^l_{t,x,v}}, 
\end{equation*} 
for all $l\geq 0$. Then, by a slight modification to the proof of \cref{lem:Em} we deduce that
\begin{align}\label{eq:emhat}
  \norm{\hat{E}_m}^2_{\mathrm{H}^l_{t,x,v}} \leq C\frac{1}{(1+m)^{2s+1}} \norm{\hat{u}_{m+1}}^2_{\mathrm{H}^{l+2+s}_{t,x,v}} \leq C\frac{1}{(1+m)^{2s+1}} \norm{u}^2_{\mathrm{H}^{l+2+s}_{t,x,v}},
\end{align}
for some constant $C=C(d,T,\norm{\nabla_x^2 U}_\infty,\dots, \norm{\nabla_x^{k+1}U}_\infty,k) \geq 1$. Finally, for $0\leq l\leq k-2$ we have
\begin{align*}
  \norm{u-u_m}^2_{\mathrm{H}^l_{t,x,v}} \leq& \norm{u-\hat{u}_m}^2_{\mathrm{H}^l_{t,x,v}} + \norm{\hat{u}_m -u_m}^2_{\mathrm{H}^l_{t,x,v}} \nonumber \\
  \leq& \left (\frac{1}{(1+m)^{k-l}} \right )^2 \norm{u}^2_{\mathrm{H}^k_{t,x,v}} + C \left (\frac{1}{(1+m)^{(k-l-2)+\frac{1}{2}}} \right )^2\norm{u}^2_{\mathrm{H}^k_{t,x,v}}\\
  \leq& C \left (\frac{1}{(1+m)^{(k-l-2)+\frac{1}{2}}} \right )^2 \left(\norm{g}^2_{\mathrm{H}^k_{x,v}} + \norm{f}^2_{\mathrm{H}^k_{t,x,v}}\right),
\end{align*}
which completes the proof of \cref{Thm1}.

\begin{remark}\label{rmk:sharp}
  The exponent $(k-l-2)+\frac{1}{2}$ appearing in \cref{Thm1} is sharp in the following sense. Fix $m$, and assume, without loss of generality,  that $d=1$. It is easy to verify that $\frac{1}{t}\norm{u}^2_{\mathrm{H}^k_{t,x,v}}\to \norm{u(0)}^2_{\mathrm{H}^k_{x,v}}$ as $t\to 0$.
  Choose the initial data $g(x, v) = g_{m+1}(x) \psi_{m+1}(v)$ where $g_{m+1}(x)$ is an arbitrary function that is sufficiently regular. Then 
  \begin{equation}\label{eq:emhatsharp}
    \norm{\sqrt{m+1}\partial_x^* g_{m+1} \psi_m }^2_{\mathrm{H}^l_{t,x,v}} \approx \norm{\frac{g}{\sqrt{1+m}}}^2_{\mathrm{H}^{l+2}_{t,x,v}} \approx \bigg (\frac{1}{(1+m)^{s+\frac{1}{2}}} \bigg )^2 \norm{g}^2_{\mathrm{H}^{l+2+s}_{x,v}}
  \end{equation}
  where $a \approx b$ means there exist some constant $C$ that $C^{-1}a\leq b \leq Ca$. Now, dividing \cref{eq:emhat} by $t$ and letting $t\to 0$, using also \cref{eq:emhatsharp}, we see the exponent $(k-l-2)+\frac{1}{2}$ is sharp.
\end{remark}

\subsection{Proof of \cref{cor:l2initial}}
Let $\{g_\varepsilon(x,v)\}$ and $\{f_\varepsilon(t,x,v)\}$, $\varepsilon > 0$, be sequences of smooth functions such that $g_\varepsilon \to g$ in $\L^2_\eta(\L^2_\mu)$ and $f_\varepsilon\to f$ in $\L^2((0,T);\L^2_\eta(\L^2_\mu))$ as $\varepsilon \to 0$.
Then by \cref{Thm1}, for each $\varepsilon > 0$ there exists a unique weak solution $u_\varepsilon$ in the sense of \cref{def:weak} to the equation $(\partial_t-\mathcal{L}) u_\varepsilon =  f_\varepsilon$ with initial data given by $g_\varepsilon$. It follows that
\begin{align}\label{energy}
  \partial_t \frac{1}{2} \norm{u_\varepsilon}^2_{\L^2_\eta(\L^2_\mu)} \leq - \norm{\nabla_v u_\varepsilon}^2_{\L^2_\eta(\L^2_\mu)} + \frac{1}{2}\norm{u_\varepsilon}^2_{\L^2_\eta(\L^2_\mu)} + \frac{1}{2}\norm{f_\varepsilon}^2_{\L^2_\eta(\L^2_\mu)}.
\end{align}
Using Grönwall's inequality we have
\begin{align*}
  \max_{0\leq t\leq T} \|u_\varepsilon\|^2_{\L^2_\eta(\L^2_\mu)}
  \lesssim \norm{g}^2_{\L^2_\eta(\L^2_\mu)}+ \norm{f}^2_{\L^2((0,T);\L^2_\eta(\L^2_\mu))}.
\end{align*}
Furthermore,  integrating \cref{energy} against $t$ we get
\begin{multline*}
  \frac{1}{2}\|u_\varepsilon(T)\|^2_{\L^2_\eta(\L^2_\mu)}
  -
  \frac{1}{2}\|u_\varepsilon(0)\|^2_{\L^2_\eta(\L^2_\mu)}
  +
  \|\nabla_v u_\varepsilon\|^2_{\L^2((0,T);\L^2_\eta(\L^2_\mu))}
  \\
  \leq
  \frac{1}{2} \bigl( \norm{u_\varepsilon}^2_{\L^2((0,T);\L^2_\eta(\L^2_\mu))}
  +
  \norm{f_\varepsilon}^2_{\L^2((0,T);\L^2_\eta(\L^2_\mu))}\bigr).
\end{multline*}
Hence
\begin{equation*}
  \|\nabla_v u_\varepsilon\|^2_{\L^2((0,T);\L^2_\eta(\L^2_\mu))} \leq C(T) \bigl( \|g\|^2_{\L^2_\eta(\L^2_\mu)}+ \norm{f}^2_{\L^2((0,T);\L^2_\eta(\L^2_\mu))}\bigr).
\end{equation*}
Next, we rewrite the equation \cref{eq:kfpconj}
\begin{equation*}
  \partial_t u_\varepsilon +v\cdot \nabla_x u_\varepsilon -\nabla_x U\cdot \nabla_v u_\varepsilon = \nabla^*_v \nabla_v u_\varepsilon + f.
\end{equation*}
From the characterization of $\mathrm{H}^{-1}_\mu$, the proof is the same as in the Lebesgue setting, see for instance \cite[Sec 5.9.1]{evansPartialDifferentialEquations2010a}, it holds
\begin{multline*}
  \|\partial_t u_\varepsilon +v\cdot \nabla_x u_\varepsilon -\nabla_x U\cdot \nabla_v u_\varepsilon\|_{\L^2((0,T);\L^2_\eta(\mathrm{H}^{-1}_\mu))}
  \\
  \lesssim
  \|\nabla_v u_\varepsilon\|_{\L^2((0,T);\L^2_\eta(\L^2_\mu))}
  +
  \norm{f}_{\L^2((0,T);\L^2_\eta(\L^2_\mu))}.
\end{multline*}
Hence there exists a $u\in \mathrm{H}^1_{\mathrm{kin}}$ such that after passing to a subsequence $u_\varepsilon \rightharpoonup u$ in $\L^2_\eta(\mathrm{H}^1_\mu)$ with $\partial_t u_\varepsilon + v\cdot\nabla_x u_\varepsilon \rightharpoonup \partial_t u+ v\cdot\nabla_x u$ in $\L^2_\eta(\mathrm{H}^{-1}_\mu)$.
Uniqueness follows easily from the estimate in \cref{energy} and linearity.

\section{Proof of \cref{Thm2}}\label{sec:global}
In this section we consider $D=\R^d$ and we prove \cref{Thm2} assuming \cref{assump:global:potential} and \cref{assump:global:potential+}. To start the construction of the family of smooth periodic approximations $\{U_R\}$ of $U$, let $Q_R$ denote the cube with side length $2R$ centered at the origin. Let $\chi(r): \R \to \R$ be smooth cutoff function such that $\chi = 1$ on $[0,1]$, $0\leq \chi \leq 1$ on $(1,2)$ and $\chi = 0$ on $[2,\infty)$. For $R>0$, we introduce
\begin{equation*}
  \varphi_R(x) := \chi(\abs{x}/{2R}).
\end{equation*}
It follows that $\varphi_R = 1$ on $B_{R/2}$, $0\leq \varphi_R\leq 1$ on $B_R\setminus B_{R/2}$ and $\varphi_R = 0$ when $\abs{x} \geq R$. Using $\varphi_R$ we define
\begin{align}\label{def:ur}
  U_R (x) := U(x)\varphi_R(x), \text{ for } x\in Q_R,
\end{align}
and we extend $U_R$ to be a periodic function of period $2R$ on $\R^d$. In the following lemma we collect some elementary properties of the construction and $U_R$.
\begin{lemma}\label{lem:global:tech}
  Let $U(x)$ satisfy the assumption \cref{assump:global:potential} and let $U_R$ be given by \cref{def:ur}. Then $U_R(x)\leq U(x)$ on $\R^d$.
  In addition, there exists an $R_0 > 1$ such that for $R\geq R_0$, then
  \begin{equation*}
    \abs{\nabla_x^k U_R(x)}\leq C_k
  \end{equation*}
  on $\R^d$ for all $k\geq 2$ and for some constants $\{C_k\}$ which are independent of $R$. Furthermore, there exists a constant $C_1$ independent of $R$ such that for $R \geq R_0$
  \begin{equation*}
    \abs{\nabla_x (U-U_R)(x)}\leq C_1 \abs{\nabla_x U(x)}
  \end{equation*}
  on $\R^d \setminus B_{R_0}$.
\end{lemma}
\begin{proof} First, obviously $U_R(x)\leq U(x)$ since $\abs{\varphi_R}\leq 1$. Secondly,
  \begin{equation*}
    \nabla_x U_R(x) = \nabla_x U(x) \varphi_R(x) + \nabla_x \varphi_R(x) U(x).
  \end{equation*}
  Let $R_0$ be so large that $\mathrm{supp}(P) \subset B_{R_0/4}$ and consider $R\geq R_0$. Then $U(x) = a\abs{x}^2$ in
  $\R^d\setminus B_{R_0/2}$, and consequently
  \begin{equation*}
    \abs{\nabla_x U_R(x)} \leq \abs{\nabla_x U(x)} + C\frac{\abs{U(x)}}{R}\leq C_1 \abs{\nabla_xU(x)},
  \end{equation*}
  on $\R^d \setminus B_{R_0}$, where $C_1$ is  a constant which is independent of $R$. The proof for  higher-order derivatives is analogous.
\end{proof}

\begin{lemma}\label{lem:global:energy}
  Let $U$ satisfy the assumption in \cref{assump:global:potential}, and let $U_R$ and $R_0$ be constructed as above and in \cref{lem:global:tech}, respectively. Consider $R\geq R_0$, and let $\ud \eta_{R} := e^{-U_R(x)}\ud x$. Assume that $u=u(x)$ is function which is periodic with period $2R$ on $\R^d$ and satisfies $u\in \L^2_{\eta_R}(Q_R)$. Then $u\in \L^2_\eta(\R^d)$ and
  \begin{equation*}
    \norm{u}_{\L^2_\eta(\R^d)} \leq C\norm{u}_{\L^2_{\eta_R}(Q_R)},
  \end{equation*}
  where $C$ is a constant independent	of $R$.
\end{lemma}
\begin{proof} Note that
  \begin{align*}
    \int_{\R^d} \abs{u(x)}^2 e^{-U(x)} \, \ud x =& \sum_{z\in\Z^d} \int_{Q_R} \abs{u(x)}^2 \frac{e^{-U(x+2Rz)}}{e^{-U_R(x)}} e^{-U_R(x)} \ud x \\
    \leq & \norm{u}^2_{\L^2_{\eta_R}(Q_R)}\cdot \sum_{z\in \Z^d} \sup_{x\in Q_R} e^{-(U(x+2Rz)-U_R(x))}.
  \end{align*}
  Using this, and
  \begin{align*}
    \sum_{z\in \Z^d} \sup_{x\in Q_r} e^{-(U(x+2Rz)-U_R(x))} \leq& 1 + \sum_{\substack{z\in \Z^d, \\z\neq 0.}} \sup_{x\in Q_R} e^{-a\abs{x+2Rz}^2+U_R(x)},
  \end{align*}
  we see that it suffices to estimate the sum
  \begin{align*}
    \sum_{\substack{z\in \Z^d, \\z\neq 0.}} \sup_{x\in Q_R} e^{-a\abs{x+2Rz}^2+U_R(x)},
  \end{align*}
  for $R\geq R_0$. Using the radial symmetry and monotonicity of $U(x)$ outside $B_R$ and \cref{assump:global:potential}, we deduce
  \begin{align}\label{eq:global:expsum}
    \notag \sum_{\substack{z\in \Z^d, \\z\neq 0}} \sup_{x\in Q_R} e^{-a\abs{x+2Rz}^2+U_R(x)}
    &\leq
    \sum_{k=1}^\infty \sum_{\substack{z\in \Z^d, \\ k \leq |z| < k+1}} \sup_{x\in Q_R} e^{-a\abs{x+2Rz}^2+U_R(x)}
    \\
    &\leq
    C(P) \sum_{k=1}^\infty k^{d-1} e^{-4aR^2 (k^2-1)}.
  \end{align}
  It is now easily seen that the sum on the right hand side in \cref{eq:global:expsum} is bounded independent of $R$ as long as we choose $R\geq R_0$.
\end{proof}

In the following we enlarge, if necessary,  $R_0$ so that $\mathrm{supp}(g(x,v))\Subset B_{R_0}(0)\times\R^d$ and $\mathrm{supp}(f(t,x,v))\Subset [0,T]\times B_{R_0}(0)\times\R^d$.

Let $\mathcal{L}_R$ denote the operator $\mathcal{L}$ but with $U$ replaced by $U_R$.
By the construction outlined we are thus led, for $R\geq R_0$, to the following family of problems with periodic boundary conditions
\begin{align}\label{ivp4}
  \begin{cases}
    (\partial_t - \mathcal{L}_R)u_R = f & \text{ in } (0,T)\times Q_R\times\R^d,\\
    \hfill u_R = g & \text{ on } \{t=0\}\times Q_R\times \R^d.
  \end{cases}
\end{align}

\begin{corollary}\label{cor:global:urexist}
  Let $U$, $f$, $g$ satisfy  \cref{assump:global:potential} and \cref{assump:global:potential+},  and let $R_0$ be as above. Then, given $R\geq R_0$ there exists a unique smooth weak solution $u_R$ to the problem \cref{ivp4}. Furthermore,
  \begin{multline*}
    \norm{u_R }^2_{\L^2((0,T);\L^2_{\eta_R}(Q_R; \mathrm{H}^{4}_\mu(\R^d)))} + \norm{\nabla_x u_R }^2_{\L^2((0,T);\L^2_{\eta_R}(Q_R;\mathrm{H}^{2}_\mu(\R^d)))}\\
    \leq C\big(\norm{g}^2_{\mathrm{H}^2_{x,v}(Q_R \times \R^d)} + \norm{f}^2_{\mathrm{H}^2_{t,x,v}((0,T)\times Q_R \times \R^d)}\big),
  \end{multline*}
  for a constant $C=C(d,T,\norm{\nabla_x^2 U}_\infty, \norm{\nabla_x^3 U}_\infty,R_0)>0$, independent of $R$.
\end{corollary}
\begin{proof}
  First of all, by the compact support of $g$ and $f$ we can extend them periodically to all of $\R^{2d}$ and $[0,T] \times \R^{2d}$ respectively, thus applying \cref{Thm1,thm:regularity} we get
  \begin{multline*}
   \norm{u_R }^2_{\L^2((0,T);\L^2_{\eta_R}(Q_R; \mathrm{H}^{4}_\mu(\R^d)))} + \norm{\nabla_x u_R }^2_{\L^2((0,T);\L^2_{\eta_R}(Q_R;\mathrm{H}^{2}_\mu(\R^d)))}\\
   \leq C'\big(\norm{g}^2_{\mathrm{H}^2_{x,v}(Q_R \times \R^d)} + \norm{f}^2_{\mathrm{H}^2_{t,x,v}((0,T)\times Q_R \times \R^d)}\big),
  \end{multline*}
  for a constant $C'=C'(d,T,\norm{\nabla_x^2 U_R}_\infty, \norm{\nabla_x^3 U_R}_\infty)>0$. Retracing our estimates in \cref{thm:regularity}, we remark that $C'$ only depends on $R$ through the norms $\norm{\nabla_x^2 U}_\infty$, $\norm{\nabla_x^3 U}_\infty$. An application of \cref{lem:global:tech} now completes the proof.
\end{proof}

Consequently, combining \cref{lem:global:energy} and  \cref{cor:global:urexist} we see that
\begin{multline}\label{eq:ur:globalbnorm}
  \norm{u_R}^2_{\L^2((0,T);\L^2_\eta(\R^d;\mathrm{H}^4_\mu(\R^d)))} + \norm{\nabla_x u_R}^2_{\L^2((0,T);\L^2_\eta(\R^d;\mathrm{H}^2_\mu(\R^d)))}\\
  \leq C \big(\norm{g}^2_{\mathrm{H}^2_{x,v}(\R^{d}\times\R^d)} + \norm{f}^2_{\mathrm{H}^2_{t,x,v}((0,T)\times\R^{d}\times\R^d)}  \big),
\end{multline}
for a constant $C$ which is independent of $R$. In particular, the family $\{u_R\}$ that is uniformly bounded in the $\mathrm{H}^2_{t,x,v}((0,T)\times\R^d\times\R^d)$ norm and $u_R$ satisfies the equation
\begin{equation*}
  (\partial_t - \mathcal{L}) u_R = \nabla_x(U-U_R)\cdot\nabla_v u_R+ f.
\end{equation*}
Hence,  passing to a subsequence $u_R\rightharpoonup u\in \mathrm{H}^2_{t,x,v}((0,T)\times\R^{d}\times\R^d)$.
Furthermore, since we assume that $f$ and $g$ are smooth, by \cref{thm:regularity} we have $u\in \mathrm{H}^k_{t,x,v}((0,T)\times\R^d\times\R^d)$ for all $k\geq 2$.
In particular,  by the Sobolev embedding theorem we have that  $u$ is smooth. However, this does not necessarily imply that $u$ is in the space $\mathrm{H}^1_\mathrm{kin}((0,T)\times\R^d\times\R^d)$ as the term $\nabla_x U\cdot\nabla_v u$ cannot readily be controlled on $(0,T)\times\R^d\times\R^d$ by the $\mathrm{H}^2_{t,x,v}((0,T)\times\R^d\times\R^d)$ norm.
Therefore, proceeding as in the proof of \cref{cor:l2initial} and \cref{Thm1}, we see that  to conclude $u_R\rightharpoonup u$ in $\mathrm{H}^1_\mathrm{kin}((0,T)\times\R^d\times\R^d)$, and that $u$ is a weak solution to \cref{ivp1}, we have to prove that
\begin{align}\label{claimsu-}
  &\mbox{$\nabla_x(U-U_R)\cdot\nabla_v u_R \rightharpoonup 0$ in $\L^2((0,T);\L^2_\eta(\R^d;\L^2_\mu(\R^d)))$}.
\end{align}
Note, by construction, since $R \geq R_0$ we have that $\supp(P) \subset B_R(0)$. As such, if $(U-U_R)(x)\neq 0$, it holds
\begin{equation*}
  (U-U_R)(x)=a(1-\varphi_R(x))|x|^2.
\end{equation*}
In particular,
\begin{equation*}
  \nabla_x (U-U_R)(x)=2a(1-\varphi_R(x))x-a\nabla_x\varphi_R(x)|x|^2,
\end{equation*}
and obviously, given $x\in\R^d$, we have that
$\nabla_x (U-U_R)(x)\to 0$ pointwise as $R\to\infty$.
Using this, the fact that functions in $\L^2((0,T);\L^2_\eta(\R^d;\L^2_\mu(\R^d)))$ with compact support are dense in $\L^2((0,T);\L^2_\eta(\R^d;\L^2_\mu(\R^d)))$, and the global bound of
$\nabla_v u_R$ in \cref{eq:ur:globalbnorm}, we see that to prove \cref{claimsu-} it suffices to prove that
\begin{align}\label{claimsu}
  &\norm{\nabla_x(U-U_R)\cdot\nabla_v u_R}_{\L^2((0,T);\L^2_\eta(\R^d;\L^2_\mu(\R^d)))}
\end{align}
is uniformly bounded as a function of $R$ if $R\geq R_0$.

To prove the claim in \cref{claimsu} we introduce the measure
\begin{equation*}
  \ud\zeta = Z_\zeta^{-1}\ud\eta\ud\mu,
\end{equation*}
where $Z_\zeta$ is a normalizing factor making $\zeta$ a probability measure on $\R^d\times\R^d$. To complete the argument we will use the following lemma.
\begin{proposition}\label{prop:global:entropy}
  Let $\zeta$ be given as above. Assume that $f(x, v) \in \mathrm{H}^1_\zeta(\R^{2d})$, and that $g(x,v)$ is a function such that
  \begin{equation*}
    \iint_{\R^{2d}} e^{g^2} \, \ud\zeta <\infty.
  \end{equation*}
  Then there exists a constant $C= C(\zeta)$ such that
  \begin{equation*}
    \iint f^2 g^2 \, \ud\zeta \leq \bigl( C \bigl (\|\nabla_x f\|^2_{\L^2_\zeta}+\|\nabla_v f\|^2_{\L^2_\zeta}\bigr ) + 2\|f\|^2_{\L^2_\zeta}\bigr)\iint e^{g^2}\, \ud\zeta.
  \end{equation*}
\end{proposition}
\begin{proof} Note that by homogeneity it suffices to prove the stated inequality assuming that $\|f\|_{\L^2_\zeta}=1$. Given   $a,b\geq 0$, the following inequality is valid, see \cite{ledouxConcentrationMeasureLogarithmic1999a},
  \begin{equation*}
    a b \leq a \log a + a + e^b.
  \end{equation*}
  Using this inequality we see that
  \begin{align*}
    \iint f^2 g^2 \ud\zeta
    &\leq \iint f^2\log(f^2) \ud\zeta + 1 + \iint e^{g^2} \ud\zeta\\
    &\leq \bigl(2\iint f^2\log(f) \ud\zeta + 2\bigr)\iint e^{g^2}\, \ud\zeta,
  \end{align*}
  where we have used the fact that $e^{g^2}\geq 1$ and that $\zeta$ is a probability measure. Finally, applying the logarithmic Sobolev inequality for $\zeta$, see \cite{bakryAnalysisGeometryMarkov2014a}, we see that
  \begin{equation*}
    \iint f^2\log(f) \ud\zeta\leq  C \bigl (\|\nabla_x f\|^2_{\L^2_\zeta}+\|\nabla_v f\|^2_{\L^2_\zeta}\bigr ),
  \end{equation*}
  which completes the proof.
\end{proof}

Now, given $\delta>0$, we have
\begin{equation*}
  \norm{\nabla_x (U-U_R)\cdot\nabla_v u_R}^2_{\L^2_\eta(\R^d;\L^2_\mu(\R^d))} = Z_\zeta \iint_{\R^{2d}} \frac{\abs{\nabla_v u_R}^2}{\delta} \delta \abs{\nabla_x (U-U_R)}^2  \ud\zeta.
\end{equation*}
Letting
\begin{equation*}
  I_{R,\delta}:=\iint_{\R^{2d}} e^{\delta \abs{\nabla_x (U-U_R)}^2} \ud\zeta,
\end{equation*}
and applying \cref{prop:global:entropy} we obtain for a constant $C=C(\zeta) > 1$ that
\begin{align*}
  &
  \norm{\nabla_x (U-U_R)\cdot\nabla_v u_R}^2_{\L^2((0,T);\L^2_\eta(\R^d;\L^2_\mu(\R^d)))}
  \\
  &
  \leq
  \frac{CZ_\zeta}{\delta}\bigl ( \|\nabla_x \nabla_v u_R\|^2_{\L^2((0,T);\L^2_\zeta(\R^{2d}))}
  +
  \|\nabla_v \nabla_v u_R\|^2_{\L^2((0,T);\L^2_\zeta(\R^{2d}))}\bigr )I_{R,\delta}
  \\
  &
  +
  2\frac{Z_\zeta}{\delta}\|\nabla_v u_R\|^2_{\L^2((0,T);\L^2_\zeta(\R^{2d}))}I_{R,\delta}
  .
\end{align*}
Thus by the above and \cref{eq:ur:globalbnorm} we have for a constant $C > 1$ independent of $R$, that
\begin{multline*}
  \norm{\nabla_x (U-U_R)\cdot\nabla_v u_R}^2_{\L^2((0,T);\L^2_\eta(\R^d;\L^2_\mu(\R^d)))}
  \\
  \leq
  C \frac{I_{R,\delta}}{\delta}
  \left (\norm{g}^2_{\mathrm{H}^2_{x,v}(\R^d \times \R^d)} + \norm{f}^2_{\mathrm{H}^2_{t,x,v}((0,T) \times \R^d \times \R^d)} \right ).
\end{multline*}
Furthermore, by \cref{lem:global:tech} we also have $\abs{\nabla_x (U-U_R)}^2 \leq C \abs{\nabla_x U}^2$ in $\R^d \setminus B_{R_0}$ for a constant $C$ independent of $R$. Thus, choosing $\delta$  small and independent of $R$, we can conclude that there exists a constant $C$ independent of $R$ such that
\begin{equation*}
  I_{R,\delta} = \iint_{\R^{2d}} e^{\delta \abs{\nabla_x (U-U_R)}^2} \ud\zeta \leq C.
\end{equation*}
This completes the proof of \cref{claimsu} and \cref{Thm2}.

\bibliography{main}

\begin{bibdiv}
\begin{biblist}

\bib{albrittonVariationalMethodsKinetic2021a}{article}{
      author={Albritton, D.},
      author={Armstrong, S.},
      author={Mourrat, J.-C.},
      author={Novack, M.},
       title={Variational methods for the kinetic {{Fokker-Planck}} equation},
        date={2021},
      eprint={arXiv:1902.04037v2},
}

\bib{adamsSobolevSpaces1975a}{book}{
      author={Adams, Robert~A.},
       title={Sobolev spaces},
      series={Pure and Applied Mathematics},
   publisher={Academic Press [Harcourt Brace Jovanovich, Publishers], New York-London},
        date={1975},
      volume={Vol. 65},
      review={\MR{450957}},
}

\bib{blaustein2023discrete}{article}{
      author={Blaustein, Alain},
      author={Filbet, Francis},
       title={On a discrete framework of hypocoercivity for kinetic equations},
        date={2023},
     journal={Math. Comp.},
       pages={to appear},
}

\bib{bakryAnalysisGeometryMarkov2014a}{book}{
      author={Bakry, Dominique},
      author={Gentil, Ivan},
      author={Ledoux, Michel},
       title={Analysis and geometry of {M}arkov diffusion operators},
      series={Grundlehren der mathematischen Wissenschaften [Fundamental Principles of Mathematical Sciences]},
   publisher={Springer, Cham},
        date={2014},
      volume={348},
        ISBN={978-3-319-00226-2; 978-3-319-00227-9},
         url={https://doi.org/10.1007/978-3-319-00227-9},
      review={\MR{3155209}},
}

\bib{bony2022eyringkramers}{article}{
      author={Bony, Jean-Francois},
      author={Peutrec, Dorian~Le},
      author={Michel, Laurent},
       title={Eyring-kramers law for fokker-planck type differential operators},
        date={2022},
      eprint={arXiv:2201.01660},
}

\bib{brinkmanBrownianMotionField1956}{article}{
      author={Brinkman, H.C.},
       title={Brownian motion in a field of force and the diffusion theory of chemical reactions. ii},
        date={1956},
        ISSN={0031-8914},
     journal={Physica},
      volume={22},
      number={1},
       pages={149\ndash 155},
         url={https://www.sciencedirect.com/science/article/pii/S0031891456800190},
}

\bib{caoExplicitConvergenceRate2023}{article}{
      author={Cao, Yu},
      author={Lu, Jianfeng},
      author={Wang, Lihan},
       title={On {E}xplicit {$L^2$}-{C}onvergence {R}ate {E}stimate for {U}nderdamped {L}angevin {D}ynamics},
        date={2023},
        ISSN={0003-9527,1432-0673},
     journal={Arch. Ration. Mech. Anal.},
      volume={247},
      number={5},
       pages={90},
         url={https://doi.org/10.1007/s00205-023-01922-4},
      review={\MR{4632836}},
}

\bib{chaiMixedGeneralizedHermiteFourier2018}{article}{
      author={Chai, Guo},
      author={Wang, Tian-jun},
       title={Mixed generalized {H}ermite-{F}ourier spectral method for {F}okker-{P}lanck equation of periodic field},
        date={2018},
        ISSN={0168-9274,1873-5460},
     journal={Appl. Numer. Math.},
      volume={133},
       pages={25\ndash 40},
         url={https://doi.org/10.1016/j.apnum.2017.10.006},
      review={\MR{3843751}},
}

\bib{dolbeault2015hypocoercivity}{article}{
      author={Dolbeault, Jean},
      author={Mouhot, Cl\'{e}ment},
      author={Schmeiser, Christian},
       title={Hypocoercivity for linear kinetic equations conserving mass},
        date={2015},
        ISSN={0002-9947,1088-6850},
     journal={Trans. Amer. Math. Soc.},
      volume={367},
      number={6},
       pages={3807\ndash 3828},
         url={https://doi.org/10.1090/S0002-9947-2015-06012-7},
      review={\MR{3324910}},
}

\bib{evansPartialDifferentialEquations2010a}{book}{
      author={Evans, Lawrence~C.},
       title={Partial differential equations},
     edition={Second},
      series={Graduate Studies in Mathematics},
   publisher={American Mathematical Society, Providence, RI},
        date={2010},
      volume={19},
        ISBN={978-0-8218-4974-3},
         url={https://doi.org/10.1090/gsm/019},
      review={\MR{2597943}},
}

\bib{fokCombinedHermiteSpectralfinite2002a}{article}{
      author={Fok, Johnson C.~M.},
      author={Guo, Benyu},
      author={Tang, Tao},
       title={Combined {H}ermite spectral-finite difference method for the {F}okker-{P}lanck equation},
        date={2002},
        ISSN={0025-5718,1088-6842},
     journal={Math. Comp.},
      volume={71},
      number={240},
       pages={1497\ndash 1528},
         url={https://doi.org/10.1090/S0025-5718-01-01365-5},
      review={\MR{1933042}},
}

\bib{gambaGalerkinPetrovApproach2018}{article}{
      author={Gamba, Irene~M.},
      author={Rjasanow, Sergej},
       title={Galerkin-{P}etrov approach for the {B}oltzmann equation},
        date={2018},
        ISSN={0021-9991,1090-2716},
     journal={J. Comput. Phys.},
      volume={366},
       pages={341\ndash 365},
         url={https://doi.org/10.1016/j.jcp.2018.04.017},
      review={\MR{3800686}},
}

\bib{gradKineticTheoryRarefied1949}{article}{
      author={Grad, Harold},
       title={On the kinetic theory of rarefied gases},
        date={1949},
        ISSN={0010-3640,1097-0312},
     journal={Comm. Pure Appl. Math.},
      volume={2},
       pages={331\ndash 407},
         url={https://doi.org/10.1002/cpa.3160020403},
      review={\MR{33674}},
}

\bib{guoCompositeGeneralizedLaguerreLegendre2009a}{article}{
      author={Guo, Ben-Yu},
      author={Wang, Tian-Jun},
       title={Composite generalized {L}aguerre-{L}egendre spectral method with domain decomposition and its application to {F}okker-{P}lanck equation in an infinite channel},
        date={2009},
        ISSN={0025-5718,1088-6842},
     journal={Math. Comp.},
      volume={78},
      number={265},
       pages={129\ndash 151},
         url={https://doi.org/10.1090/S0025-5718-08-02152-2},
      review={\MR{2448700}},
}

\bib{herauTunnelEffectKramersFokkerPlanck2008}{article}{
      author={H\'{e}rau, Fr\'{e}d\'{e}ric},
      author={Hitrik, Michael},
      author={Sj\"{o}strand, Johannes},
       title={Tunnel effect for {K}ramers-{F}okker-{P}lanck type operators: return to equilibrium and applications},
        date={2008},
        ISSN={1073-7928,1687-0247},
     journal={Int. Math. Res. Not. IMRN},
      number={15},
       pages={Art. ID rnn057, 48},
         url={https://doi.org/10.1093/imrn/rnn057},
      review={\MR{2438070}},
}

\bib{herauTunnelEffectSymmetries2011a}{article}{
      author={H\'{e}rau, Fr\'{e}d\'{e}ric},
      author={Hitrik, Michael},
      author={Sj\"{o}strand, Johannes},
       title={Tunnel effect and symmetries for {K}ramers-{F}okker-{P}lanck type operators},
        date={2011},
        ISSN={1474-7480,1475-3030},
     journal={J. Inst. Math. Jussieu},
      volume={10},
      number={3},
       pages={567\ndash 634},
         url={https://doi.org/10.1017/S1474748011000028},
      review={\MR{2806463}},
}

\bib{herauIsotropicHypoellipticityTrend2004a}{article}{
      author={H\'{e}rau, Fr\'{e}d\'{e}ric},
      author={Nier, Francis},
       title={Isotropic hypoellipticity and trend to equilibrium for the {F}okker-{P}lanck equation with a high-degree potential},
        date={2004},
        ISSN={0003-9527,1432-0673},
     journal={Arch. Ration. Mech. Anal.},
      volume={171},
      number={2},
       pages={151\ndash 218},
         url={https://doi.org/10.1007/s00205-003-0276-3},
      review={\MR{2034753}},
}

\bib{herauSemiclassicalAnalysisKramersFokkerPlanck2005}{article}{
      author={H\'{e}rau, Fr\'{e}d\'{e}ric},
      author={Sj\"{o}strand, Johannes},
      author={Stolk, Christiaan~C.},
       title={Semiclassical analysis for the {K}ramers-{F}okker-{P}lanck equation},
        date={2005},
        ISSN={0360-5302,1532-4133},
     journal={Comm. Partial Differential Equations},
      volume={30},
      number={4-6},
       pages={689\ndash 760},
         url={https://doi.org/10.1081/PDE-200059278},
      review={\MR{2153513}},
}

\bib{ledouxConcentrationMeasureLogarithmic1999a}{incollection}{
      author={Ledoux, Michel},
       title={Concentration of measure and logarithmic {S}obolev inequalities},
        date={1999},
   booktitle={S\'{e}minaire de {P}robabilit\'{e}s, {XXXIII}},
      series={Lecture Notes in Math.},
      volume={1709},
   publisher={Springer, Berlin},
       pages={120\ndash 216},
         url={https://doi.org/10.1007/BFb0096511},
      review={\MR{1767995}},
}

\bib{liHermiteSpectralMethod2021a}{article}{
      author={Li, Ruo},
      author={Ren, Yinuo},
      author={Wang, Yanli},
       title={Hermite spectral method for {F}okker-{P}lanck-{L}andau equation modeling collisional plasma},
        date={2021},
        ISSN={0021-9991,1090-2716},
     journal={J. Comput. Phys.},
      volume={434},
       pages={Paper No. 110235, 22},
         url={https://doi.org/10.1016/j.jcp.2021.110235},
      review={\MR{4224112}},
}

\bib{meyerCommentsGradProcedure1983b}{article}{
      author={Meyer, J.},
      author={Schr\"{o}ter, J.},
       title={Comments on the {G}rad procedure for the {F}okker-{P}lanck equation},
        date={1983},
        ISSN={0022-4715,1572-9613},
     journal={J. Statist. Phys.},
      volume={32},
      number={1},
       pages={53\ndash 69},
         url={https://doi.org/10.1007/BF01009419},
      review={\MR{712228}},
}

\bib{pavliotisStochasticProcessesApplications2014a}{book}{
      author={Pavliotis, Grigorios~A.},
       title={Stochastic processes and applications},
      series={Texts in Applied Mathematics},
   publisher={Springer, New York},
        date={2014},
      volume={60},
        ISBN={978-1-4939-1322-0; 978-1-4939-1323-7},
         url={https://doi.org/10.1007/978-1-4939-1323-7},
        note={Diffusion processes, the Fokker-Planck and Langevin equations},
      review={\MR{3288096}},
}

\bib{riskenFokkerplanckEquationMethods1996}{book}{
      author={Risken, H.},
       title={The {F}okker-{P}lanck equation},
     edition={Second},
      series={Springer Series in Synergetics},
   publisher={Springer-Verlag, Berlin},
        date={1989},
      volume={18},
        ISBN={3-540-50498-2},
         url={https://doi.org/10.1007/978-3-642-61544-3},
        note={Methods of solution and applications},
      review={\MR{987631}},
}

\bib{sarnaConvergenceAnalysisGrad2020}{article}{
      author={Sarna, Neeraj},
      author={Giesselmann, Jan},
      author={Torrilhon, Manuel},
       title={Convergence analysis of {G}rad's {H}ermite expansion for linear kinetic equations},
        date={2020},
        ISSN={0036-1429,1095-7170},
     journal={SIAM J. Numer. Anal.},
      volume={58},
      number={2},
       pages={1164\ndash 1194},
         url={https://doi.org/10.1137/19M1270884},
      review={\MR{4083350}},
}

\bib{schmeiserConvergenceMomentMethods1998}{article}{
      author={Schmeiser, Christian},
      author={Zwirchmayr, Alexander},
       title={Convergence of moment methods for linear kinetic equations},
        date={1999},
        ISSN={0036-1429,1095-7170},
     journal={SIAM J. Numer. Anal.},
      volume={36},
      number={1},
       pages={74\ndash 88},
         url={https://doi.org/10.1137/S0036142996304516},
      review={\MR{1654590}},
}

\bib{tolleUniquenessWeightedSobolev2012b}{article}{
      author={T\"{o}lle, Jonas~M.},
       title={Uniqueness of weighted {S}obolev spaces with weakly differentiable weights},
        date={2012},
        ISSN={0022-1236,1096-0783},
     journal={J. Funct. Anal.},
      volume={263},
      number={10},
       pages={3195\ndash 3223},
         url={https://doi.org/10.1016/j.jfa.2012.08.002},
      review={\MR{2973338}},
}

\bib{villaniHypocoercivity2009}{article}{
      author={Villani, C\'{e}dric},
       title={Hypocoercivity},
        date={2009},
        ISSN={0065-9266,1947-6221},
     journal={Mem. Amer. Math. Soc.},
      volume={202},
      number={950},
       pages={iv+141},
         url={https://doi.org/10.1090/S0065-9266-09-00567-5},
      review={\MR{2562709}},
}

\bib{wienerFourierIntegralCertain1988}{book}{
      author={Wiener, Norbert},
       title={The {F}ourier integral and certain of its applications},
      series={Cambridge Mathematical Library},
   publisher={Cambridge University Press, Cambridge},
        date={1988},
        ISBN={0-521-35884-1},
         url={https://doi.org/10.1017/CBO9780511662492},
        note={Reprint of the 1933 edition, With a foreword by Jean-Pierre Kahane},
      review={\MR{983891}},
}

\end{biblist}
\end{bibdiv}

\end{document}